\documentclass[10 pt, leqno]{article}
\date{}
\usepackage{amssymb, amsbsy, amsmath, amsfonts, amssymb, amscd, mathrsfs, amsthm, dsfont,enumerate}

\usepackage[english]{babel}
\usepackage[T1]{fontenc}
\usepackage{indentfirst}
\usepackage{color}

\makeatletter
\@addtoreset{equation}{section}
\makeatother

\newtheorem{statement}{}[section]
\newtheorem{theorem}[statement]{Theorem}
\newtheorem{lemma}[statement]{Lemma}

\newtheorem{proposition}[statement]{Proposition}
\newtheorem{definition}[statement]{Definition}
\newtheorem{corollary}[statement]{Corollary}

\newcommand\C{\mathbb C}

\newcommand\R{\mathbb R}
\newcommand\T{\mathbb T}
\newcommand\D{\mathbb D}

\newcommand\e{{\rm e}}

\newcommand\eps{\varepsilon}
\newcommand\ind{\mathds{1}}
\newcommand\dis{\displaystyle}
\renewcommand \Re{{\mathfrak R}{\rm e}\,}
\renewcommand \Im{{\mathfrak I}{\rm m}\,}
\newcommand\converge{\mathop{\longrightarrow}\limits}

\let\phi=\varphi
\let\hat = \widehat
\let\tilde=\widetilde
\newcommand\tq{\, ; \ }
\newcommand{\overbar}[1]{\mkern 1.5mu\overline{\mkern-1.5mu#1\mkern-1.5mu}\mkern 1.5mu}

\title{\bf Characterization of weighted Hardy spaces  on which all composition operators are bounded}
\author{\it Pascal~Lef\`evre, Daniel~Li,  \\ \it Herv\'e~Queff\'elec, Luis~Rodr{\'\i}guez-Piazza}

\date{\footnotesize \today}

\begin{document}

\maketitle

\noindent {\bf Abstract.} We give a complete characterization of the sequences $\beta = (\beta_n)$ of positive numbers for which all composition operators on 
$H^2 (\beta)$ are bounded, where $H^2 (\beta)$ is the space of analytic functions $f$ on the unit disk $\D$  such that 
$\sum_{n = 0}^\infty |a_n|^2 \beta_n < + \infty$ if $f (z) = \sum_{n = 0}^\infty a_n z^n$. 
We prove that all composition operators are bounded on $H^2 (\beta)$ if and only if $\beta$ is essentially decreasing and slowly oscillating.
We also prove that every automorphism of the unit disk induces a bounded composition operator on $H^2 (\beta)$ if and only if $\beta$ is slowly oscillating. 
We give applications of our results.
\medskip

\noindent {\bf MSC 2010.} primary: 47B33 ; secondary: 30H10
\smallskip

\noindent {\bf Key-words.} automorphism of the unit disk; composition operator; slowly oscillating sequence; weighted Hardy space 

\section {Introduction} \label{sec: intro}

Let $\beta = (\beta_n)_{n\geq 0}$ be a sequence of positive numbers such that
\begin{equation} \label{unus} 
\liminf_{n \to \infty} \beta_n^{1/n} \geq 1 \, .
\end{equation} 
The associated weighted Hardy space $H^{2} (\beta)$ is the Hilbertian space of analytic functions $f (z) = \sum_{n = 0}^\infty a_n z^n$  such that 
 \begin{equation} \label{duo} 
\Vert f \Vert^2 := \sum_{n = 0}^\infty |a_n|^2 \beta_n < \infty \, . 
\end{equation}

Condition (\ref{unus}) is equivalent to the inclusion $H^{2} (\beta) \subseteq {\cal H}ol (\D)$. Indeed, if  \eqref{unus} holds, we have 
$H^2 (\beta) \subseteq {\cal H}ol (\D)$ since $|a_n|^2\beta_n$  is bounded and thanks to the Hadamard formula. Conversely, testing the inclusion 
$H^2 (\beta)\subseteq {\cal H}ol (\D)$ on the function $f(z)=\sum_{n=1}^\infty \frac{1}{n\sqrt{\beta(n)}}z^n \in H^2 (\beta) $, we get \eqref{unus} 
from the Hadamard formula. 
\smallskip

Condition \eqref{unus} will therefore be assumed throughout this paper, without repeating it.
\smallskip

When $\beta_n \equiv 1$, we recover the usual Hardy space $H^2$; the Bergman space corresponds to $\beta_n = 1 / (n + 1)$, 
and the Dirichlet space to $\beta_n = n + 1$.
\smallskip

Recall that a symbol is a (non constant) analytic self-map $\phi \colon \D \to \D$, and the associated composition operator 
$C_\phi \colon H^2 (\beta) \to {\cal H}ol (\D)$ is defined as: 
\begin{equation} 
C_\phi (f) = f \circ \phi \, .
\end{equation} 

An important question in the theory is to decide when $C_\phi$ is bounded on $H^2 (\beta)$, i.e. when  $C_\phi \colon H^2 (\beta) \to H^2 (\beta)$. 

This question appears in the literature in several places. 
For instance, it is Problem 1 in the thesis of Nina Zorboska \cite[Pb~1, p.\,49]{ZorboPgD}. 
This thesis contains many interesting results, in particular Proposition~\ref{simpleBoundVan} and Proposition~\ref{tous pour un} of the present paper (actually 
we discovered the content of Zorboska's thesis once the present paper was almost finished).
See also Question 36 raised by Deddens in \cite[p.\,122.c]{Shields}.
\smallskip

When $H^2 (\beta)$ is the usual Hardy space $H^2$ (i.e. when $\beta_n \equiv 1$), it is well-known, as a consequence of the Littlewood subordination principle 
(see \cite{Littlewood}), that all symbols generate bounded composition operators (see \cite[pp.~13--17]{Shapiro-livre}). On the other hand, for the Dirichlet space, 
corresponding to $\beta_n = n + 1$, not all composition operators are bounded since there exist symbols $\varphi$ not belonging to the Dirichlet space 
(e.g. any infinite Blaschke product). 

Note that, by definition of the norm of $H^2 (\beta)$, all rotations $R_\theta$, defined by $R_\theta(z)=\e^{i\theta}z$, with $\theta \in \R$, induce bounded and 
surjective composition operators on $H^2 (\beta)$ and send isometrically $H^2 (\beta)$ into itself.  
\smallskip

Our goal in this paper is characterizing the sequences $\beta$ for which all composition operators act boundedly on the space $H^2 (\beta)$, i.e. 
send $H^2 (\beta)$ into itself. 
\smallskip

In Shapiro's presentation for the Hardy space $H^2$, the main point is the case $\phi (0) = 0$ and a subordination principle for subharmonic functions 
(Littlewood's subordination principle). 
The case of automorphisms is claimed simple, using an integral representation for the norm and some change of variable. 
For general weights $\beta$, the situation is different, as we will see in this paper, and it turns out that the conditions on $\beta$ for the boundedness of the 
composition operators $C_\phi$ on $H^2 (\beta)$ are not the same depending on whether we consider the class of all symbols such that $\phi (0) = 0$, or the 
class of symbols $\phi = T_a$, where 
\begin{equation} 
T_ a (z) = \frac{a + z}{1+ \bar{a} \, z} 
\end{equation} 
for $a \in \D$. 

It is clear that when these two classes of composition operators are bounded, then all composition operators are bounded. 
Recall that every symbol $\varphi$ can be written as the composition $\varphi=T_a\circ\psi$ where $\psi(0)=0$ and $a=\varphi(0)$; and then 
$C_\phi = C_\psi \circ C_{T_a}$.
\smallskip

In many occurrences, the weight $\beta$ is defined as 
\begin{equation} \label{integral beta}
\beta_n = \int_{0}^1 t^n \, d \sigma (t) 
\end{equation} 
where $\sigma$ is a positive measure on $(0, 1)$; more specifically the following definition is often used: 
let $G \colon (0, 1) \to \R_+$ be an integrable function and let $H^2_G$ be the space of analytic functions $f \colon \D \to \C$ such that:
\begin{equation} \label{particular case}
\Vert f \Vert_{H^2_G}^{2} := \int_{\D} |f (z)|^2 \, G (1 - |z|^2) \, dA (z) < \infty \, .
\end{equation} 

Such weighted Bergman type spaces are used, for instance, in  \cite{Karim-Pascal}, \cite{Kriete-MacCluer} and in \cite{LQR-radius}. We have 
$H^2_G = H^2 (\beta)$ with:
\begin{equation} \label{coto}
\beta_n = 2 \int_{0}^{1} r^{2 n + 1} G (1 - r^2) \, dr = \int_{0}^{1} t^{n} \, G (1 - t ) \, dt \, ,
\end{equation} 
and the sequence $\beta = (\beta_n)_n$ is \emph{non-increasing} (actually, the above representation \eqref{integral beta} is equivalent, by the Hausdorff 
moment theorem, to a high regularity of the sequence $\beta$, namely its \emph{complete monotony}). 

When the weight $\beta$ is non-increasing (or more generally, essentially decreasing), all the symbols vanishing at the origin induce a bounded composition 
operator. This was proved by C.~Cowen \cite[Corollary, page~31]{Cowen}, using Hadamard multiplication. 
We can also use Kacnel'son's theorem (see \cite{Isabelle} or \cite[Theorem~3.12]{LLQR-comparison}). 
Actually that follows from an older theorem of Goluzin \cite {Goluzin} (see \cite[Theorem~6.3]{Duren}), which itself uses  a self-refinement observed by 
Rogosinski of Littlewood's principle (\cite[Theorem~6.2]{Duren}). 

For weights defined as in \eqref{integral beta}, we dispose of integral representations for the norm in $H^2 (\beta)$, and, as in the Hardy space case, this integral 
representation rather easily allows us to decide when the boundedness of $C_{T_a}$ on $H^2 (\beta)$ occurs.
This is not always the case, as shown by T.~Kriete and B.~MacCluer in \cite{Kriete-MacCluer}. 
They consider spaces of Bergman type $A^2_{\tilde G} := H^2_G$, where $\tilde G (r) = G (1 - r^2)$, defined as the spaces of analytic functions in $\D$ such 
that $\int_\D | f (z)|^2 \, {\tilde G} (|z|) \, dA < \infty$, for a positive non-increasing continuous function ${\tilde G}$ on $[0, 1)$. 
They prove \cite[Theorem~3]{Kriete-MacCluer} that for
\begin{displaymath} 
\qquad \qquad \qquad {\tilde G} (r) = \exp \, \bigg( - B \, \frac{1}{(1 - r)^\alpha} \bigg) \, , \qquad B > 0 \, , \ 0 < \alpha \leq 2 \, ,
\end{displaymath} 
and 
\begin{displaymath} 
\qquad \qquad \qquad \qquad \phi (z) = z + t (1 - z)^\gamma \, , \qquad \qquad 1 <\gamma \leq 3 \, , \ 0 < t < 2^{1 - \gamma} \, ,
\end{displaymath} 
then $\phi$ is a symbol and $C_\phi$ is bounded on $A^2_{\tilde G}$ if and only if $\gamma \geq \alpha + 1$.

Here
\begin{displaymath} 
\beta_n = \int_0^1 t^n \e^{- B / (1 - \sqrt{t})^\alpha} \, dt \lesssim \exp ( - c \, n^{\alpha / (\alpha  +1)} ) \, .
\end{displaymath} 
We point out that $\beta$ is non-increasing, so for every symbol $\phi$ fixing the origin, the composition operator $C_\phi$ is bounded. 
Nevertheless, choosing $\gamma < \alpha + 1$, there exist symbols inducing an unbounded composition operator, hence not all the $C_{T_a}$ are bounded. 
Actually, for every $\alpha \in(0,2]$, no $C_{T_a}$ is bounded because $\beta$ has no polynomial lower estimate (see Proposition~\ref{polynomial lb necessary} 
below).
\medskip

Let us describe the contents of this paper. In Section~\ref{sec: def}, we introduce several notion of growth or regularity for a sequence $\beta$: essentially 
decreasing, polynomial decay and polynomial growth, slow oscillation, and give some connections 
between them. 
In Section~\ref{sec: phi (0) = 0}, we consider the composition operators whose symbol vanishes at the origin. 
We show that in order for all these operators to be bounded, it is necessary that $\beta$ be bounded above. 
We show that $\beta$ is essentially decreasing if and only if all these operators are bounded and $\sup_{\phi (0) = 0} \| C_\phi \| < + \infty$. 
In Theorem~\ref{theo cond suff}, we give a sufficient condition for having all the composition operators $C_\phi$ with $\phi (0) = 0$ bounded, allowing us to 
give an example of a sequence $\beta$ for which this happens though 
$\sup_{\phi (0) = 0} \| C_\phi \| = + \infty$ (Theorem~\ref{theo example}). 
In Section~\ref{sec: T_a}, we prove that all $C_{T_a}$ are bounded on $H^2 (\beta)$ if and only if $\beta$ is slowly oscillating 
(Theorem~\ref{theo CS T_a beta SO} and Theorem~\ref{theo CN T_a beta SO}). We hence have:

\begin{theorem} \label{main theorem completed}
Let $\beta$ be a sequence of positive numbers, and let 
\begin{displaymath}
T_ a (z) = \frac{a + z}{1+ \bar{a} \, z}
\end{displaymath}
for $a \in \D$. The following assertions are equivalent:

\begin{itemize}
\setlength\itemsep {-0.1 em}

\item [$1)$] for some $a\in\D\setminus\{0\}$, the map $T_a$ induces a bounded composition operator $C_{T_a}$ on $H^2 (\beta)$;

\item [$2)$] for all $a \in \D$, the maps $T_a$ induce bounded composition operators $C_{T_a}$ on $H^2(\beta)$;

\item [$3)$] $\beta$ is slowly oscillating.
\end{itemize}
\end{theorem}

The deep implication is $2)\Rightarrow3)$. Its proof requires some sharp estimates on the mean of Taylor coefficients of $T_a$ for $a$ belonging to a subinterval 
of $(0,1)$. Once we found the equivalence of $1)$ and $2)$, we realized that it already appeared in Zorboska's thesis \cite{ZorboPgD}.

In Section~\ref{all symbols}, we show (Theorem~\ref{example without ED}) that if $\beta$ is slowly oscillating, and moreover all composition operators are 
bounded on $H^2 (\beta)$, then $\beta$ is essentially decreasing. We thus obtain the following theorem. 
\goodbreak

\begin{theorem} \label{full characterization}
Let $\beta$ be a sequence of positive numbers.  
Then all composition operators on $H^2 (\beta)$ are bounded if and only if $\beta$ is essentially decreasing and slowly oscillating. 
\end{theorem}

For the notion of essentially decreasing and slowly oscillating sequences, see Definition~\ref{def: essentially decreasing} and Definition~\ref{def: SO}
\medskip

We end the paper with some results about multipliers.

\medskip

A first version of this paper, not including the complete characterization given here, was put on arXiv on 30 November 2020 
(and a second version on 21 March 2022) under the title \emph{Boundedness of composition operators on general weighted Hardy spaces of analytic functions}.

\section {Definitions, notation, and preliminary results} \label{sec: def}

The open unit disk of $\C$ is denoted $\D$ and we write $\T$ its boundary $\partial \D$. We set $e_n (z) = z^n$, $n \geq 0$.
\smallskip

The weighted Hardy space $H^2(\beta)$ defined in the introduction is a Hilbert space with the canonical orthonormal basis 
\begin{equation} 
\qquad \qquad e^\beta_{n} (z) = \frac{1}{\sqrt{\beta_n}} \, z^n \, , \quad  n\geq 0 \, ,
\end{equation} 
and the reproducing kernel $K_w$ given for all $w\in \D$ by 
 \begin{equation} \label{tres} 
K_{w} (z) = \sum_{n = 0}^\infty e^\beta_{n} (z) \, \overbar{e^\beta_{n} (w)} 
= \sum_{n = 0}^\infty \frac{1}{\beta_n} \,\overbar{w}^n \, z^n \, .
\end{equation}

Note that $H^2$ is continuously embedded in $H^2 (\beta)$ if and only if $\beta$ is bounded above. In particular, this is the case when $\beta$ is 
non-increasing. In this paper, we need a slightly more general notion.

\begin{definition} \label{def: essentially decreasing}
A sequence of positive numbers $\beta = (\beta_n)_{n\geq 0}$ is said to be \emph{essentially decreasing} if, for some constant $C \geq 1$, we have, for all 
$m \geq n \geq 0$:
\begin{equation} 
\beta_m \leq C \, \beta_n \, .
\end{equation} 
\end{definition}

Note that saying that $\beta$ is essentially decreasing means that the shift operator on $H^2 (\beta)$ is power bounded.
\smallskip

If $\beta$ is essentially decreasing, and if we set:
\begin{displaymath} 
\tilde \beta_n = \sup_{m \geq n} \beta_m \, ,
\end{displaymath} 
the sequence $\tilde \beta = (\tilde\beta_n)$ is non-increasing and we have $\beta_n \leq \tilde \beta_n \leq C \, \beta_n$. 
In particular, $H^2(\beta)=H^2(\tilde \beta)$ (with equivalent norms) and $H^2$ is continuously embedded in $H^2 (\beta)$. 

\begin{definition} \label{def: SO} 
We say that a sequence $\beta$ is \emph{slowly oscillating} if there are positive constants $c < 1 < C$ such that 
\begin{equation}
\qquad c \leq \frac{\beta_m}{\beta_n} \leq C \quad \text{when } n / 2 \leq m \leq 2 n \, . 
\end{equation}
\end{definition}
We may remark  that this is equivalent to the existence of some function $\rho \colon (0, \infty) \to (0, \infty)$ which is bounded above on each compact 
subset of $(0, \infty)$ and for which $\beta_m / \beta_n \leq \rho (m / n)$, equivalently $$\frac{1}{\rho(n/m)} \leq\frac{\beta_m}{\beta_n}\leq \rho (m / n).$$

\begin{definition} \label{def: polynomial minoration}
The sequence of positive numbers $\beta = (\beta_n)$ is said to have polynomial decay if there are positive constants $c$ and $\alpha$ 
such that, for all integers $n \geq 1$:
\begin{equation}\label{slow}  
\beta_{n} \geq c \, n^{- \alpha} \, .
\end{equation} 
\end{definition}
That means that $H^2 (\beta)$ is continuously embedded in the weighted Bergman space ${\mathfrak B}^2_{\alpha-1}$ of the analytic functions 
$f \colon \D \to \C$ such that 
\begin{displaymath} 
\| f \|_{{\mathfrak B}^2_{\alpha-1}}^2 := \alpha\int_\D |f (z)|^2 (1 - |z|^2)^{\alpha-1} \, dA (z) < \infty
\end{displaymath} 
since ${\mathfrak B}^2_{\alpha-1} = H^2(\gamma)$ with $\gamma_n \approx n^{-\alpha}$.

\begin{definition} \label{def: polynomial majoration}
The sequence of positive numbers $\beta = (\beta_n)$ is said to have \emph{polynomial growth} if  there are positive constants $C$ and $\gamma$ 
such that, for all integers $n \geq 1$:
\begin{equation}\label{eq: polyn maj}  
\beta_{n} \leq C \, n^\gamma \, .
\end{equation} 
\end{definition}
\medskip

The following simple proposition links those notions.

\begin{proposition} \label{simple proposition polynomial minoration} \hfill \par

$1)$ Every slowly oscillating sequence $\beta$ has polynomial decay and growth.

$2)$ There are sequences that are essentially decreasing, and with polynomial decay, but not slowly oscillating.

$3)$ There are bounded sequences that are slowly oscillating, but not essentially decreasing.
\end{proposition}
\begin{proof} 
$1)$ This is clear, because, for some $c\in(0,1)$, if $2^j \leq n < 2^{j + 1}$, then
\begin{displaymath} 
\beta_n \geq c \, \beta_{2^j} \geq c^{j + 1} \beta_1 \geq c\, \beta_1 \, n^{- \alpha} \, ,
\end{displaymath} 
with $\alpha = \log (1 / c) / \log 2$; and, for some $C>1$,
\begin{displaymath}
\beta_n \leq C \, \beta_{2^j} \leq C^{j + 1} \beta_1 \leq C\beta_1 \, n^\gamma \, ,
\end{displaymath}
with $\gamma = \log C / \log 2$.
\smallskip

$2)$ Let $\delta > 0$. We set $\beta_0 = \beta_1 = 1$ and for $n \geq 2$:
\begin{displaymath} 
\beta_n = \frac{1}{(k!)^\delta} \quad \text{when } k! < n \leq (k + 1)! \, .
\end{displaymath} 
The sequence $\beta$ is non-increasing.

For $n$ and $k$ as above, we have:
\begin{displaymath} 
\beta_n = \frac{1}{(k!)^\delta} \geq  \frac{1}{n^\delta} \, ;
\end{displaymath} 
hence $\beta$ has arbitrarily slow polynomial decay. However we have, for $k \geq 2$:
\begin{displaymath} 
\frac{\beta_{2 (k!)}}{\beta_{k!}} = \frac{(k!)^{- \delta}}{[(k - 1)!]^{- \delta}} = \frac{1}{k^{\delta}} \converge_{k \to \infty} 0 \, ,
\end{displaymath} 
so $\beta$ is not slowly oscillating.
\smallskip

$3)$ We define $\beta_n$ as follows. Let $(a_k)$ be an increasing sequence of positive square integers such that 
$\lim_{k \to \infty} a_{k + 1} / a_k = \infty$, for example $a_k = 4^{k^2}$, and let $b_k = \sqrt{a_k a_{k + 1}}\,$; with our choice, this is an integer and 
we clearly have $a_k < b_k < a_{k + 1}$. We set:
\begin{align*}
\beta_n = 
\left\{ 
\begin{array} {ll}
\, a_k / n & \text{for } a_k \leq n < b_k \\
& \\
(a_k / b_k^2) \, n = (1 / a_{k + 1}) \, n & \text{for } b_k \leq n < a_{k + 1} \, .
\end{array}
\right.
\end{align*}
This sequence $(\beta_n)$ is slowly oscillating by construction. 
Indeed, since the other cases are obvious, it suffices to check that for $a_k \leq n / 2 < b_k \leq n < a_{k + 1}$, the quotient $\beta_m / \beta_n$ remains lower 
and upper bounded when $n / 2 \leq m \leq n$ (it will then be automatically also satisfied when $n\leq m\le 2n$). But for $n / 2 \leq m < b_k$, we have
\begin{displaymath} 
\frac{\beta_m}{\beta_n} = \frac{a_k / m}{n / a_{k + 1}} = \frac{a_k a_{k + 1}}{m n} = \frac{b_k^2}{m n} \, \raise 1 pt \hbox{,}
\end{displaymath} 
which is $\leq 2 \, b_k^2 / n^2 \leq 2$ and $\geq b_k^2 / n^2 \geq (n / 2)^2 / n^2 = 1 / 4$; and for $b_k \leq m$, we have
\begin{displaymath} 
\frac{\beta_m}{\beta_n} = \frac{m / a_{k + 1}}{n / a_{k + 1}} = \frac{m}{n} \in \Big[ \frac{1}{2} \, \raise 1,5 pt \hbox{,} \, \raise 0.5 pt \hbox{$1$} \Big] \, .
\end{displaymath} 

However, though $(\beta_n)$ is bounded, since $\beta_n \leq 1$ for $a_k \leq n < b_k$ and, 
for $b_k \leq n < a_{k + 1}$, 
\begin{displaymath} 
\beta_n \leq \beta_{a_{k + 1} - 1} = \frac{1}{a_{k + 1}}\, (a_{k + 1} - 1) \leq 1 \, ,
\end{displaymath} 
it is not essentially decreasing, since
\begin{displaymath} 
\frac{\beta_{a_{k + 1} - 1}}{\beta_{b_k}} = \frac{1}{\sqrt{a_k a_{k + 1}}} \, (a_{k + 1} - 1)  \sim \sqrt{\frac{a_{k + 1}}{a_k}} 
\converge_{k \to \infty} \infty \, . \qedhere
\end{displaymath} 
\end{proof}
\bigskip

Now we are going to recall some well known facts about matrix representation of an operator $T$ defined on a Hilbert space with an orthonormal basis 
$(\e_n)_{n\ge0}$ and how it is translated in our framework.

The entry $a_{m,n}$ (where $m,n\ge0$) is defined by the $m^{th}$ coordinate of $T(\e_n)$:
$$a_{m,n}=\e_m^\ast\big(T(\e_n)\big)$$
where $\e_k^\ast(x)$ stands for the $k^{th}$ coordinate of the vector $x$.

We shall use the notation $\widehat{f}(k)$ for the $k^{th}$ Fourier coefficient of a function $f\in L^1(-\pi,\pi)$:
$$\widehat{f}(k)=\frac{1}{2\pi}\int_{-\pi}^\pi f(t)\,\e^{-ikt}\,dt\;.$$

Let us point out that when the operator is the composition operator $C_\phi$ associated to the symbol $\phi$, viewed on $H^2(\beta)$, its matrix representation 
in the basis $\big(e^\beta_{n}\big)_{n\ge0}$ has an entry $(m,n)$ which can be written as
$$(e^\beta_{m})^\ast\big(C_\phi(e^\beta_{n})\big) 
= \frac{\sqrt{\beta_m}}{\sqrt{\beta_n}}e_{m}^\ast\big(\phi^{n}\big)=\frac{\sqrt{\beta_m}}{\sqrt{\beta_n}}\,\widehat{\phi^n}(m)$$
%
since the $m^{th}$ Taylor coefficient of $\phi^n$ coincides with its $m^{th}$ Fourier coefficient.
\smallskip

We say that the reproducing  kernels $K_w$ have a \emph{slow growth} if 
\begin{equation} 
\| K_w \| \leq \frac{C}{(1 - |w|)^s} 
\end{equation} 
for positive constants $C$ and $s$. We have the following equivalence.

\begin{proposition} \label{reproducing kernel polynomial decay} 
The sequence $\beta$ has polynomial decay if and only if the reproducing kernels $K_w$ of $H^2 (\beta)$ have a slow growth.
\end{proposition}
\begin{proof}
Assume that the reproducing kernels have a slow growth. Since 
\begin{displaymath}
\| K_w \|^2 = \sum_{k = 0}^\infty \frac{\ | w |^{2 k}}{\beta_k} \, \raise 1 pt \hbox{,}
\end{displaymath}
we get, for any $k \geq 2$:
\begin{displaymath}
\frac{\ | w |^{2 k}}{\beta_k} \leq  \frac{C^2}{(1 - | w |)^{2 s}} \, \cdot
\end{displaymath}
Taking $w = 1 - \frac{1}{k}\,$, we obtain $\beta_k \geq C' \, k^{- 2 s}$.

For the necessity, we only have to see that:
\begin{displaymath} 
\| K_w \|^2 = \frac{1}{\beta_0} + \sum_{n = 1}^\infty \frac{\, \, \, |w|^{2 n}}{\beta_n} 
\leq \frac{1}{\beta_0} + \delta^{- 1} \sum_{n = 1}^\infty n^\alpha |w|^{2 n} \leq \frac{C}{(1 - |w|^2)^{\alpha  + 1}} \, \cdot \qedhere
\end{displaymath} 
\end{proof}
%

\section {Boundedness of composition operators whose symbol vanishes at the origin} \label{sec: phi (0) = 0} 

\subsection {Necessary conditions} 

We begin with this simple observation (see \cite[Proposition~3.1]{ZorboPgD}).

\begin{proposition}\label{simpleBoundVan}
If all composition operators with symbol vanishing at $0$ are bounded on $H^2 (\beta)$, then the sequence $\beta$ is bounded above.
\end{proposition}
\begin{proof}
Let $f \in H^\infty$.  Write $f = A \, \phi +  f (0)$ where $A$ is a constant and $\phi$ a symbol vanishing at $0$. We have $\phi = C_\phi (z) \in H^2 (\beta)$, by 
hypothesis. So that $f \in H^2 (\beta)$ and $H^\infty \subseteq H^2 (\beta)$.  It follows (by the closed graph theorem, since the convergence in norm implies 
pointwise convergence) that there exists a constant $M$ such that $\| f \|_{H^2 (\beta)} \leq M \, \| f \|_\infty$ for all $f \in H^\infty$. Testing that with 
$f (z) = z^n$, we get $\beta_n \leq M^2$. 
\end{proof}
\smallskip

Let us point out that boundedness of $\beta_n$ does not suffice. For example, let $(\beta_n)$ be a sequence such that 
 $\beta_{4 k +2} / \beta_{2 k + 1} \converge_{k \to \infty} \infty$ (for instance $\beta_{2 k} = 1$ and 
$\beta_{2 k + 1} = 1 / (k + 1)$); if $\phi (z) = z^2$, then $\| C_\phi (z^{2 n + 1}) \|^2 = \| z^{2 (2 n + 1)} \|^2 = \beta_{2 (2 n + 1)}$; since 
$\| z^{2 n + 1} \|^2 = \beta_{2 n + 1}$, the operator $C_\phi$ is not bounded on $H^2 (\beta)$. 
\medskip

A partial characterization is given in the next proposition. 

\begin{proposition} \label{phi(0)=0}
The following assertions are equivalent:
\smallskip

$1)$ all symbols $\phi$ such that $\phi (0) = 0$ induce bounded composition operators $C_\phi$ on $H^2 (\beta)$ {\bf and} 
\begin{equation} \label{borne uniforme}
\sup_{\phi (0) = 0} \| C_\phi \| < \infty \, ;
\end{equation}

$2)$ $\beta$ is an essentially decreasing sequence.
\end{proposition}

Of course, by the uniform boundedness principle, \eqref{borne uniforme} is equivalent to:
\begin{displaymath} 
\qquad \qquad \sup_{\phi (0) = 0} \| f \circ \phi \| < \infty \quad \text{for all } f \in H^2 (\beta) \, .
\end{displaymath} 

Let us point out an important fact: we shall see in Theorem~\ref{theo example} that there are weights $\beta$ for which all composition operators $C_\phi$ with 
$\phi (0) = 0$ are bounded, but $\sup_{\phi (0) = 0} \| C_\phi \| = + \infty$.

\begin{proof}
$2) \Rightarrow 1)$ We may assume that $\beta$ is non-increasing. Then the Goluzin-Rogosinski theorem (\cite[Theorem~6.3]{Duren}) gives the result; 
in fact, writing $f (z) = \sum_{n = 0}^\infty c_n z^n$ and $(C_\phi f) (z) = \sum_{n = 0}^\infty d_n z^n$, it says that

$$\sum_{0\leq k\leq n} |d_k|^2\leq  \sum_{0\leq k\leq n} |c_k|^2\quad  \forall n\geq 0,$$
and hence, by Abel summation:
\begin{displaymath} 
\| C_\phi f \|^2 = \sum_{n = 0}^\infty |d_n|^2 \beta_n \leq \sum_{n = 0}^\infty |c_n|^2 \beta_n = \| f \|^2 \, ,
\end{displaymath}
%
%
leading to $C_\phi$ bounded and $\| C_\phi \| \leq 1$. This result was also proved by C.~Cowen \cite[Corollary of Theorem~7]{Cowen}. Alternatively, 
we can use a result of Kacnel'son (\cite{Kacnelson}; see also \cite{Isabelle}, \cite[Corollary~2.2]{IsabelleII}, or \cite[Theorem~3.12]{LLQR-comparison}). 

$1) \Rightarrow 2)$ Set $M = \sup_{\phi (0) = 0} \| C_\phi \|$. Let $m > n$, and take
\begin{displaymath}
\phi (z) = \phi_{m, n} (z) = z \, \bigg( \frac{1 + z^{m - n}}{2} \bigg)^{1 / n} \, .
\end{displaymath}
Then $\phi (0) = 0$ and  $[\phi (z)]^n = \frac{z^n + z^m}{2}\,$; hence
\begin{displaymath} 
\frac{1}{4} \, (\beta_n + \beta_m)  = \| \phi^n \|^2 = \| C_\phi (e_n) \|^2 \leq \| C_\phi \|^2 \| e_n \|^2 \leq M^2 \, \beta_n \, ,
\end{displaymath} 
so $\beta$ is essentially decreasing.
\end{proof}

\noindent {\bf Remark.} Let us mention the following example. For $0 < r < 1$, let $\beta_n = \pi \,  n \, r^{2 n}$, for $n\ge1$ and $\beta_0=1$. This sequence 
is eventually decreasing, so it is essentially decreasing. The quantity $\| f \|^2_{H^2 (\beta)}-|f(0)|^2$ is the area of the part of the Riemann surface on which 
$r \D$ is mapped by $f$. E.~Reich \cite{Reich}, generalizing Goluzin's result \cite{Goluzin} (see \cite[Theorem~6.3]{Duren}), proved that for all symbols 
$\phi$ such that $\phi (0) = 0$, the composition operator $C_\phi$ is bounded on $H^2 (\beta)$ and
\begin{displaymath} 
\| C_\phi \| \leq \sup_{n \geq 1} \sqrt{n} \, r^{n - 1} \leq \frac{1}{\sqrt{2 \, \e}} \, \frac{1}{r \sqrt{\log (1 / r)}} \, \cdot
\end{displaymath} 
For $0 < r < 1 / \sqrt{2}$, Goluzin's theorem asserts that $\|C_\phi \| \leq 1$.
\smallskip

Note that this sequence $\beta$ is not slowly oscillating since $\beta_{2 n} / \beta_n = 2 \, r^{2 n}$. Hence, from 
Theorem~\ref{theo CN T_a beta SO} below, we get that no composition operator $C_{T_a}$ is bounded on $H^2(\beta)$. 
\medskip

However, that the weight $\beta$ is essentially decreasing is not necessary for the boundedness of all composition operators $C_\phi$,  
with symbol $\phi$ vanishing at $0$, as we will see later (Theorem~\ref{theo example}). 

\goodbreak
\subsection {Sufficient condition} \label{subsec: CS phi (0) = 0} 

\begin{theorem} \label{theo cond suff}
Let $\beta = (\beta_n)_{n = 0}^\infty$ be a sequence of positive numbers with polynomial decay and growth.
Assume that $\beta$ is \emph{weakly decreasing}, i.e.:
\begin{equation} \label{cond Luis}
\begin{split}
\text{For every } \delta > 0, & \text{ there exists } \text{a positive constant } C = C (\delta) \text{ such that } \\
& \hskip - 0.8 em \beta_m  \leq C \, \beta_n \quad \text{whenever } m > (1 + \delta) \, n \, .
\end{split}
\end{equation}

Then, for all symbols $\phi \colon \D \to \D$ vanishing at $0$, the composition operator $C_\phi$ is bounded on $H^2 (\beta)$.
\end{theorem}
Let us point out that \eqref{cond Luis} implies that $\beta$ is bounded. 
For a counterexample with an exponential weight, see \cite[Ex.~1, p.~14-15]{ZorboPgD}.
\smallskip

To prove Theorem~\ref{theo cond suff}, we need several lemmas.

\begin{lemma} \label{lemma 1}
Let $\phi \colon \D \to \D$ be an analytic self-map such that $\phi (0) = 0$ and $| \phi ' (0) | < 1$. Then there exists $\rho > 0$ such that, for every integers 
$n$ and $m$, 
\begin{displaymath}
|\hat{\phi^n} (m) | \leq \exp \Big( - \frac{1}{2} \, [(1 + \rho) \, n - m] \Big) \, .
\end{displaymath}
\end{lemma}
\begin{proof}
Since $\phi (0) = 0$, we can write $\phi (z) = z\, \phi_1 (z)$. Since $| \phi ' (0) | < 1$, we have $\phi_1 \colon \D \to \D$. Let $M (r) = \sup_{|z| = r} | \phi_1 (z) |$. 
Cauchy's inequalities say that $|\hat {\phi_1^n} (m)| \leq [M (r)]^n / r^m$. We have $M (r) < 1$, so there exists a positive number $\rho = \rho (r)$ such that 
$M (r) = r^\rho$. We get:
\begin{displaymath}
|\hat {\phi^n} (m)| = |\hat {\phi_1^n} (m - n)| \leq \frac{r^{\rho n}}{r^{m - n}} = r^{(1 + \rho)\, n - m} \, ,
\end{displaymath}
and the result follows, by taking $r = \e^{- 1 / 2}$.
\end{proof}

The next lemma is a variant of the following result of V.~\`E~Kacnel'son (\cite{Kacnelson}; see also \cite{Isabelle}, \cite[Corollary~2.2]{IsabelleII}, 
or \cite[Theorem~3.12]{LLQR-comparison}).
\begin{theorem} [V.~\`E.~Kacnel'son] \label{theo Kacnelson} 
Let $H$ be a separable complex Hilbert space and $(e_i)_{i \geq 0}$ a fixed orthonormal basis of $H$. 
Let $M \colon H \to H$ be a bounded linear operator. We assume that the matrix of $M$ with respect to this basis is lower-triangular: 
$\langle M e_j \mid e_i \rangle = 0$ for $i < j$.

Let $(\gamma_j)_{j \geq 0}$ be a non-decreasing sequence of positive real numbers and $\Gamma$ the (possibly unbounded) diagonal operator such that 
$\Gamma (e_j) = \gamma_j e_j$, $j \geq 0$. Then the operator $\Gamma^{ - 1} M \, \Gamma \colon H \to H$ is bounded and moreover:
\begin{displaymath} 
\| \Gamma^{- 1} M \, \Gamma \| \leq \| M \| \, .
\end{displaymath}
\end{theorem}

This variant is used implicitly in \cite[page~13]{LLQR-comparison}.

\begin{lemma} \label{lemma 2}
Let $A \colon \ell_2 \to \ell_2$ be a bounded operator represented by the matrix $\big( a_{m, n} \big)_{m, n}$, i.e. 
$a_{m, n} = \langle A \, e_n, e_m \rangle$, where $(e_n)_{n \geq0}$ is the canonical basis of $\ell_2$. 

Let $(d_n)$ be a sequence of positive numbers such that, for every $m$ and $n$:
\begin{equation} \label{cond trou}
d_m < d_n \quad \Longrightarrow \quad a_{m, n} = 0 \, .
\end{equation}

Then, $D$ being the (possibly unbounded) diagonal operator with entries $d_n$, we have:
\begin{displaymath}
\| D^{- 1} A D \| \leq \| A \| \, .
\end{displaymath}
\end{lemma}

The proof is the same as that of Kacnel'son's theorem, but we reproduce it for the convenience of the reader.
Actually we propose two different proofs.
\begin{proof}[Proof 1.]
Let $\C_0$ be the right-half plane $\C_0 = \{z \in \C \tq \Re z > 0 \}$. We set $H_N = {\rm span}\, \{ e_n \tq n \leq N \}$ and
\begin{displaymath}
A_N = P_N A J_N \, ,
\end{displaymath}
where $P_N$ is the orthogonal projection from $\ell_2$ onto $H_N$ and $J_N$ the canonical injection from $H_N$ into $\ell_2$. 
We consider, for $z \in \overbar{\C_0}$:
\begin{displaymath}
A_N (z) = D^{- z} A_N D^z \colon H_N \to H_N \, ,
\end{displaymath}
where $D^z (e_n) = d_n^{\, z} e_n$. 

If $\big( a_{m, n} (z) \big)_{m, n}$ is the matrix of $A_N (z)$ on the basis $\{e_n \tq n \leq N\}$ of $H_N$, we clearly have:
\begin{displaymath}
a_{m, n} (z) = a_{m, n} (d_n / d_m)^z \, .
\end{displaymath}
In particular, we have, thanks to \eqref{cond trou}:
\begin{displaymath}
a_{m, n} (z) = 0  \quad \text{if } d_m < d_n \, ,
\end{displaymath}
and
\begin{displaymath}
\qquad\qquad\quad | a_{m, n} (z) | \leq \sup_{k, l} |a_{k, l}| := M \, , \qquad \text{for all } z \in \overbar{\C_0} \, .
\end{displaymath}
Since $\| A_N (z) \|^2 \leq \| A_N (z)\|_{HS}^2 = \sum_{m, n \leq N} |a_{m, n} (z)|^2 \leq (N + 1)^2 M^2$, we get:
\begin{displaymath}
\qquad \qquad \| A_N (z) \| \leq (N + 1) \, M \qquad \text{for all } z \in \overbar{\C_0} \, .
\end{displaymath}

Let us consider the function $u \colon \overbar{\C_0} \to \overbar{\C_0}$ defined by:
\begin{equation}
u_N (z) = \| A_N (z) \| \, .
\end{equation}
This function $u_N$ is continuous on $\overbar{\C_0}$, bounded above by $(N + 1) M$, and subharmonic in $\C_0$. Moreover, thanks to 
\eqref{cond trou}, the maximum principle gives:
\begin{displaymath}
\sup_{\overbar{\C_0}} u_N (z) = \sup_{\partial \C_0} u_N (z) \, .
\end{displaymath}
Since $\| D^z \| = \| D^{- z} \| = 1$ for $z \in \partial \C_0$, we have $\| A_N (z) \| \leq \| A_N \|$ for $z \in \partial \C_0$, and we get:
\begin{displaymath}
\sup_{\overbar{\C_0}} u_N (z) \leq \| A_N \| \leq \| A \| \, .
\end{displaymath}
In particular $u_N (1) \leq \| A \|$, and, letting $N$ going to infinity, we obtain that $\| D^{- 1} A D \| \leq \| A \|$.
\medskip

\emph{Proof 2.}
Since $d_n$ is positive, we can write $d_n=\e^{-\rho_n}$ where $\rho_n\in \mathbb{R}$. 
Let $x=(x_n)_{n\ge0}$ and $y=(y_n)_{n\ge0}\in\ell^2$ with finite support, we are interested in controlling the sum
 
$$S=\sum_{m,n} a_{m,n}\frac{d_n}{d_m}\, x_n \overline{y_m}$$
which can also be written

$$ S=\sum_{m,n} a_{m,n}  \e^{-|\rho_n-\rho_m|}\,x_n \overline{y_m}.$$ 
since the non trivial part of the sum runs over the pairs $(m,n)$ such that $d_m\geq d_n$ {\it i.e.} $\rho_n\geq\rho_m$.

Now we introduce the function $\dis f(t)=\frac{1}{\pi(1+t^2)}$ for $t\in\R$, which is positive and belongs to the unit ball of $L^1(\R)$.
Moreover, its Fourier transform satisfies, for every $x\in\R$,
$$\mathcal{F}(f)(-x)=\int_\R f(t)\e^{ixt}dt=\e^{-|x|}\;.$$

We get
$$S=\int_\R f(t)\bigg(\sum_{m,n} a_{m,n} x_n\e^{i\rho_{n}t} \overline{y_m \e^{i\rho_{m}t}}\bigg) \, dt 
= \int_\R f(t) \langle A(x(t)), y(t)\rangle_{\ell^2}\, dt$$
where 
$$x(t)=\big(x_n \e^{i\rho_{n}t}\big)_{n\ge0}\quad\text{and}\quad\ y(t) = \big(y_n \e^{i\rho_{n}t}\big)_{n\ge0}.$$
We obtain  
$$|S|\leq \int_\R f(t) \Vert A\Vert \,\Vert x(t)\Vert\, \Vert y(t)\Vert \, dt 
= \int_\R f(t) \Vert A\Vert\, \Vert x\Vert\, \Vert y\Vert \, dt = \Vert A\Vert \Vert x\Vert \Vert y\Vert$$ 
since $\|f\|_{L^1(\R)}=1$. 

Since $x$ and $y$ are arbitrary, this proves $\dis\Vert D^{-1}AD\Vert \leq \Vert A\Vert$. 
\end{proof}
\smallskip

\begin{proof} [Proof of Theorem~\ref{theo cond suff}]
First, if $|\phi ' (0) | = 1$, we have $\phi (z) = \alpha \, z$ for some $\alpha$ with $|\alpha | = 1$, and the result is trivial.

So, we assume that $|\phi ' (0) | < 1$. Then, by Lemma~\ref{lemma 1}, there exists $\rho > 0$ such that, for all $m$, $n$:
\begin{displaymath}
| \hat{\phi^n} (m) | \leq \exp \Big( - \frac{1}{2} \, [(1 + \rho) \, n - m] \Big) \, .
\end{displaymath}

Since $\phi (0) = 0$, we also know that $\hat{\phi^n} (m) = 0$ if $m < n$. 
\smallskip

Take $\delta = \rho / 2$ and use property \eqref{cond Luis}: there exists $M\ge1$ such that:
\begin{displaymath}
\qquad \frac{\beta_m}{\beta_n} \leq M \quad \text{when } m \geq (1 + \delta) \, n \, .
\end{displaymath}

Define now a new sequence $\gamma = (\gamma_n)$ as:
\begin{displaymath}
\gamma_n = \max \bigg\{ \beta_n, \sup_{m > (1 + \delta) \, n} \beta_m \bigg\} \, .
\end{displaymath}

We have:
\smallskip

1) $\beta_n \leq \gamma_n \leq M \, \beta_n$;
\smallskip

2) $\gamma_m \leq \gamma_n$ \quad if $m \geq (1 + \delta) \, n$.
\medskip

Item 1) implies that $H^2 (\gamma) = H^2 (\beta)$, and we are reduced to prove that 
$C_\phi \colon H^2 (\gamma) \to H^2 (\gamma)$ is bounded.
\smallskip

Let $A = \big( a_{m, n} \big)_{m, n} = \big( \hat{\phi^n} (m) \big)_{m, n}$. We have to prove that
\begin{displaymath}
B = \big( \gamma_m^{1 / 2} \gamma_n^{- 1 / 2} a_{m, n} \big)_{m, n}
\end{displaymath}
represents a bounded operator on $\ell_2$.
\smallskip

Define the matrix
\begin{displaymath}
A_1 = \big( a_{m, n} \ind_{\{ (m, n) \tq m \leq (1 + \delta) \, n \} } \big)_{m, n}
\end{displaymath}
and set $A_2 = A - A_1$. Define analogously $B_1$ and $B_2 = B - B_1$.

Then $A_1$ is a Hilbert-Schmidt operator, because (recall that $a_{m, n} = 0$ if $m < n$), we have
\begin{align*}
\sum_{n = 1}^\infty \sum_{m = 1}^{(1 + \delta) \, n} |a_{m, n} |^2 
& \leq \sum_{n = 1}^\infty \sum_{m = n}^{(1 + \delta) \, n} \exp \big( - [(1 + \rho) \, n - m] \big) \\
& \leq \sum_{n = 1}^\infty (\delta \, n + 1) \, \exp ( - \delta \, n ) < \infty \, .
\end{align*}

Now, $\beta$ is bounded and has polynomial decay, so, for some positive constants $C_1$, $C_2$, and $\alpha$, we have:
\begin{align*}
\sum_{n = 1}^\infty \sum_{m = n}^{(1 + \delta) \, n} \frac{\gamma_m}{\gamma_n} \, |a_{m, n} |^2 
& \leq \sum_{n = 1}^\infty \sum_{m = n}^{(1 + \delta) \, n} \frac{ C_1}{n^{- \alpha}} \, \exp ( - \delta \, n) \\
& \leq \sum_{n = 1}^\infty C_2 (1 + \delta) n^{\alpha  + 1} \exp ( - \delta \, n) < \infty \, ,
\end{align*}
meaning that $B_1$ is a Hilbert-Schmidt operator.
\smallskip

Since $A$ is bounded, it follows that $A_2 = A - A_1$ is bounded. Remark that, writing $A_2 = \big( \alpha_{m, n} \big)_{m, n}$, we have, with 
$d_n = 1 / \sqrt{\gamma_n}$:
\begin{displaymath}
d_m < d_n \quad \Longrightarrow \quad \gamma_m > \gamma_n \quad \Longrightarrow \quad m < (1 + \delta) \, n \quad \Longrightarrow 
\quad \alpha_{m, n} = 0 \, .
\end{displaymath}
Hence we can apply Lemma~\ref{lemma 2} to the matrix $A_2$, and it ensues that $B_2$ is bounded, and therefore that $B = B_1 + B_2$ is 
bounded as well, as wanted.
\end{proof}
\medskip

As a corollary of Theorem~\ref{theo cond suff}, we can provide the following example.

\begin{theorem} \label{theo example}
There exists a bounded sequence $\beta$, with polynomial decay, but which is \emph{not essentially decreasing}, and for which every composition 
operator with symbol vanishing at $0$ is bounded on $H^2 (\beta)$. 
\smallskip

We hence have $\sup_{\phi (0) = 0} \| C_\phi \| = + \infty$.
\end{theorem}

It should be noted that for this weight, the composition operators are not all bounded, as we will see in Proposition~\ref{T_a not bounded}.

\begin{proof} 
Define $\beta_n = 1$ for $n \leq 3!$, and, for $k \geq 3$:
\begin{displaymath}
\left\{
\begin{array} {lcl} 
\beta_n = & \dis \frac{1}{k!} & \text{for } k! < n \leq (k + 1)! - 2 \text{ and for } n = (k + 1)! \\
\\
\beta_n = & \dis \frac{1}{(k + 1)!} & \text{for } n = (k + 1)! - 1 \, .
\end{array}
\right.
\end{displaymath}

Note that, for $m > n$, we have $\beta_m > \beta_n$ only if $n = (k + 1)! - 1$ and $m = (k + 1)! = n + 1$, for some $k \geq 3$. 
\smallskip

However $\beta$ is not essentially decreasing since, for every $k \geq 3$, we have $\beta_{n + 1} / \beta_n = k + 1$ if $n = (k + 1)! - 1$. 
\smallskip

The sequence $\beta$ has a polynomial decay because $\beta_n \geq 1 / (2 \, n)$ for all $n \geq 1$. In fact, for $k \geq 3$, we have 
$\beta_n \geq (k + 1)/ n \geq 1 / n$ if $k! < n \leq (k + 1)! - 2$ or if $n = (k + 1)!$; and for $n = (k + 1)! - 1$, we have 
$n \, \beta_n = [(k + 1)! - 1] / (k + 1)! \geq 1 /2$. It has a polynomial growth since it is bounded above, by $1$.
\smallskip

Now, it remains to check \eqref{cond Luis} in order to apply Theorem~\ref{theo cond suff} and finish the proof of Theorem~\ref{theo example}. 
Note first that we have $\beta_m / \beta_n \leq 1$ if $m \geq n + 2$. Next, for given $\delta > 0$, there exists an integer $N$ such that 
$(1 + \delta) \, n \geq n + 2$ for every $n \geq N$, so $\beta_m / \beta_n \leq 1$ if $m \geq (1 + \delta)\, n$ and $n \geq N$. 
It suffices to take $C = \max_{1 \leq n \leq N} \beta_{n + 1} / \beta_n$ to obtain \eqref{cond Luis}. The last assertion follows from Proposition~\ref{phi(0)=0}.
\end{proof}

\goodbreak

\section {Boundedness of composition operators of symbol $T_a$} \label{sec: T_a}

Recall that for $a \in \D$, we defined
\begin{equation} 
\qquad \quad T_a (z) = \frac{a + z}{1 + \bar{a} \, z} \, \raise 1 pt \hbox{,} \quad z \in \D \, .
\end{equation} 

It is well-known that $T_a$ is an automorphism of $\D$ and that $T_a (0) = a$ and $T_a (- a) = 0$.

Though we do not really need this, we may remark that $(T_a)_{a \in (- 1, 1)}$ is a group and $(T_a)_{a \in (0, 1)}$ is a semigroup. 
It suffices to see that $T_a \circ T_b = T_{a \ast b}$, with:
\begin{equation} \label{produit}
a \ast b = \frac{a + b}{1 + a b} \, \cdot 
\end{equation} 

In this section, we are going to prove a necessary and sufficient condition in order that all composition operators $C_{T_a}$ for $a \in \D$ are bounded on 
$H^2 (\beta)$. Namely, we have the following theorem, the proof of which will occupy Section~\ref{subsec: CS for T_a} and Section~\ref{subsec: CN for T_a}. 

\begin{theorem} \label{theo CNS for T_a}
All composition operators $C_{T_a}$, with $a \in \D$ are bounded on $H^2 (\beta)$ if and only if $\beta$ is slowly oscillating. 
\end{theorem}

Before that, let us note the following fact (see also \cite[Proposition~3.6]{ZorboPgD}). Recall that if $\phi$ and $\psi$ are two symbols, then 
$C_\phi \circ C_\psi = C_{\psi \circ \phi}$. 

\begin{proposition} \label{tous pour un}
If $C_{T_a}$ is bounded on $H^2 (\beta)$ for some $a \in \D \setminus \{0\}$, then $C_{T_b}$ is bounded on $H^2 (\beta)$ for all $b \in \D$. 

Moreover the maps $C_{T_b}$ are uniformly bounded on the compact subsets of $\D$.
\end{proposition}

 We decompose the proof into lemmas. The first one was first proved in \cite{ZorboPgD} (see also \cite[Proposition~2.1]{Eva-Jonathan}), and follows from 
the fact that if $b = \rho \, \e^{i \theta}$ and $R_\theta$ is the rotation $R_\theta (z) = \e^{i \theta} z$, which induces a unitary operator 
$C_{R_\theta}$ on $H^2 (\beta)$, then $T_b = R_\theta \circ T_\rho \circ R_{- \theta}$ and 
$C_{T_b} = C_{R_{- \theta}} \circ C_{T_\rho} \circ C_{R_\theta}$. 

\begin{lemma} \label{first lemma}
The composition operator $C_{T_b}$ is bounded if and only if $C_{T_{|b|}}$ is bounded, with equal norms.
\end{lemma}

\begin{lemma}\label{second lemma} Let $r\in(0,1)$ such that $ C_{T_r} $ is bounded. 
For any $b\in\D$ satisfying $| b | \leq \frac{2 r}{1 + r^2},$ $C_{T_b}$ is bounded and we have $\| C_{T_b} \| \leq \| C_{T_r} \|^2$.
\end{lemma}

\begin{proof} Let $S$ be the circle $C(0, r)$ and $u \colon S \to \R_+$ be the continuous function defined by
\begin{equation}
u (s) = \bigg| \frac{s + r}{1 + \overline{s}r} \bigg| \, \cdot
\end{equation}
By connectedness, $u (S)$  contains the segment $\big[0, \frac{2 r}{1 + r^2} \big] = [ u (- r), u (r)]$. 
Let now $b \in D \big(0, \frac{2 r}{1 + r^2} \big)$. By the above, there exists $s \in S$ such that $| b | = u (s)$. That means that  
$| T_b (0) | =|b|=|u(s)|= |T_s (r)| = | (T_s \circ T_r) (0) |$. Therefore, $T_b (0) = \e^{i \alpha} (T_s \circ T_r)(0)$ for some $\alpha \in \R$, and hence, by 
Schwarz's lemma, there is some $\theta \in \R$ such that $T_b = R_\alpha\circ T_s \circ T_r\circ R_\theta$. We then have 
$C_{T_b} =C_{R_\theta}\circ C_{T_r} \circ C_{T_s} \circ  C_{R_\alpha}$. Since $C_{R_\theta}$ and $ C_{R_\alpha}$ are unitary, we get using 
Lemma~\ref{first lemma} for $C_{T_s}$:
\begin{displaymath}
\| C_{T_b} \| = \| C_{T_r} \circ C_{T_s} \| \leq \| C_{T_r} \| \, \| C_{T_s} \| = \| C_{T_r} \|^2 \, . \qedhere
\end{displaymath}
\end{proof} 

\begin{proof} [Proof of Proposition~\ref{tous pour un}] 
It suffices to use Lemma~\ref{first lemma} and Lemma~\ref{second lemma} and do an iteration, in noting that if $r_0 = |a| > 0$ and 
$r_{n + 1} = \frac{2 r_n}{1 + r_n^2} = r_n \ast r_n$, then $(r_n)_{n \geq 0}$ increases to $1$.
\end{proof}

\goodbreak

\subsection {An elementary necessary condition} 

We begin by an elementary necessary condition. It is implied by Theorem~\ref{theo CN T_a beta SO}, but its statement deserves to be pointed out. 
Moreover, its proof is simple and highlights the role of the reproducing kernel. 

\begin{proposition} \label{polynomial lb necessary}
Let $a \in (0, 1)$ and assume that $T_a$ induces a bounded composition operator on $H^2 (\beta)$. Then $\beta$ has polynomial decay.
\end{proposition}
\begin{proof} 
Since 
\begin{displaymath} 
\| K_x \|^2 = \sum_{n = 0}^\infty \frac{\ x^{ 2 n}}{\beta_n} \,  \raise 1 pt \hbox{,}
\end{displaymath} 
we have $\| K_x \| \leq \| K_y \|$ for $0 \leq x \leq y < 1$.
\smallskip

We define by induction a sequence $(u_n)_{n \geq 0}$ with:
\begin{displaymath}
u_0 = 0 \qquad \text{and} \qquad u_{n + 1} = T_a (u_n) \, .
\end{displaymath}
Since $T_a (1) = 1$ (recall that $a \in (0, 1)$), we have:
\begin{displaymath}
1 - u_{n + 1} = \int_{u_n}^1 T_a ' (t) \, dt = \int_{u_n}^1 \frac{1 - a^2}{(1 + a t)^2} \, dt \, ;
\end{displaymath}
hence 
\begin{displaymath}
\frac{1 - a}{1 + a} \, (1 - u_n) \leq 1 - u_{n + 1} \leq (1 - a^2) (1 - u_n) \, .
\end{displaymath}
Let $0 < x < 1$. We can find $N \geq 0$ such that $u_N \leq x < u_{N  + 1}$. Then:
\begin{displaymath}
1 - x \leq 1 - u_N \leq (1 - a^2)^N \, .
\end{displaymath}

On the other hand, since $C_{T_a}^{^{\, \ast}} K_z = K_{T_a (z)}$ for all $z \in \D$, we have:
\begin{displaymath}
\| K_x \| \leq \| K_{u_{N  + 1}} \| \leq \| C_{T_a} \| \, \| K_{u_N} \| \leq \| C_{T_a} \|^{N  + 1} \| K_{u_0} \| 
= \frac{1}{\sqrt{\beta_0}}\| C_{T_a} \|^{N  + 1} \, .
\end{displaymath}
Let $ s \geq 0$ such that $(1 - a^2)^{- s} = \| C_{T_a} \|$. We obtain:
\begin{equation} \label{slow growth of reproducing kernels}
\| K_x \| \leq\frac{1}{\sqrt{\beta_0}(1 - x)^s}\| C_{T_a} \| \, .
\end{equation}
We get the result by using Proposition~\ref{reproducing kernel polynomial decay}.
\end{proof}

\noindent {\bf Remarks.} 1) For example, when $\beta_n = \exp \big[ - c \, \big( \log (n + 1) \big)^2 \big]$, with $c > 0$, no $T_a$ induces a bounded 
composition operator on $H^2 (\beta)$, though $C_\phi$ is bounded for all symbols $\phi$ with $\phi (0) = 0$, since $\beta$ is decreasing, as we saw in 
Proposition~\ref{phi(0)=0}.\smallskip

2) For the Dirichlet space ${\cal D}^2$, we have $\beta_n = n + 1$, but all the maps $T_a$ induce bounded composition operators on ${\cal D}^2$ 
(see \cite[Remark before Theorem~3.12]{LLQR-comparison}). In this case $\beta$ has polynomial growth though it is not bounded above. 
\smallskip

3) However, even for decreasing sequences, a polynomial decay for $\beta$ is not enough for some $T_a$ to induce a bounded composition operator. 
Indeed, we saw in Proposition~\ref{simple proposition polynomial minoration} an example of a decreasing sequence $\beta$  with polynomial decay but not 
slowly oscillating, and we will see in Theorem~\ref{theo CN T_a beta SO} that this condition is needed for having 
some $T_a$ inducing a bounded composition operator.
\smallskip

4) In \cite{Eva-Jonathan}, Eva Gallardo-Guti\'errez and Jonathan Partington give estimates for the norm of $C_{T_a}$, with $a \in (0, 1)$, when $C_{T_a}$ is 
bounded on $H^2 (\beta)$. More precisely, they proved that if $\beta$ is bounded above and $C_{T_a}$ is bounded, then
\begin{displaymath} 
\| C_{T_a} \| \geq \bigg( \frac{1 + a}{1 - a} \bigg)^\sigma \, \raise 1 pt \hbox{,}
\end{displaymath} 
where $\sigma = \inf \{s \geq 0 \tq (1 - z)^{- s} \notin H^2 (\beta) \}$, and 
\begin{displaymath} 
\| C_{T_a} \| \leq \bigg( \frac{1 + a}{1 - a} \bigg)^\tau \, \raise 1 pt \hbox{,}
\end{displaymath} 
where $\tau = \frac{1}{2} \, \sup \Re W (A)$, with $A$ the infinitesimal generator of the continuous semigroup $(S_t)$ defined as $S_t = C_{T_{\tanh t}}$, 
namely $(Af) (z) = f ' (z) (1 - z^2)$, and $W (A)$ its numerical range. 

For $\beta_n = 1 / (n + 1)^\nu$ with $0 \leq \nu \leq 1$, the two bounds coincide, so they get 
$\| C_{T_a} \| = \big( \frac{1 + a}{1 - a} \big)^{(\nu + 1)/2}$. 

\goodbreak

\subsection {Sufficient condition} \label{subsec: CS for T_a}

The following sufficient condition explains in particular why all composition operators $C_{T_a}$ are bounded on the Dirichlet space. 

\begin{theorem} \label{theo CS T_a beta SO}
If $\beta$ is slowly oscillating, then all symbols that extend analytically in a neighborhood of $\overbar{\D}$ induce a bounded composition operator on 
$H^2 (\beta)$. 

In particular, all $C_{T_a}$ for $a \in \D$ are bounded on $H^2 (\beta)$. 
\end{theorem}

To prove Theorem~\ref{theo CS T_a beta SO}, we begin by a very elementary fact. 

\begin{lemma} \label{majo coeff}
Let $\phi: \D\to \D$ have an  analytic extension to an open neighborhood $\Omega$ of $\overbar{\D}$.
Then, there are a constant $b > 0$ and an integer $\lambda>1$ such that 
\begin{displaymath}
| \hat{\phi^n} (m) | \leq \left\{ 
\begin{array} {lcl}
\e^{- b n} & \text{if} & n \geq\lambda m \, , \\
\e^{- b m} & \text{if} & m \geq\lambda n \, . 
\end{array}
\right.
\end{displaymath}
\end{lemma}

\begin{proof}
Let $R > 1$ such that $\overline{D (0, R)} \subseteq \Omega$. For $0 < r \leq R$, we set
\begin{displaymath} 
M (r) = \sup_{|z| = r}  |\phi (z)| \, . 
\end{displaymath} 
Take any $r \in(0, 1)$, for instance $r=\e^{-1}$. We have $M (r) < 1$, so we can write $M (r) = \e^{-\rho}$, for some positive $\rho$.

Cauchy's inequalities give:
\begin{displaymath} 
|\hat{\phi^n} (m) | \leq \frac{[M(r)]^n}{r^m} = \e^{m-\rho n} \, .
\end{displaymath} 
Choose $\lambda_1=\max(2,2/\rho)$  and $b_1=\rho-\lambda_1^{-1}$ then $|\hat{\phi^n} (m) | \leq\e^{- b_1 n}$ if $n \geq\lambda_1 m$.

For the second inequality, write $R =: \e^{\beta}$, with $\beta > 0$.  Let $\alpha>0$ with $M (R) \leq \e^{\alpha}.$  
Cauchy's inequalities again give:
\begin{displaymath} 
| \hat{\phi^n} (m)| \leq \frac{[M (R)]^n}{R^m} \leq \e^{\alpha n - \beta m} \, .
\end{displaymath}

Choose $\lambda_2=\max(2,2\alpha/\beta)$  and $b_2=\beta-\alpha\lambda_2^{-1}$ then $|\hat{\phi^n} (m) | \leq\e^{- b_2 m}$ if $m \geq\lambda_2 n$.
We get the conclusion taking $b=\min(b_1,b_2)$ and choosing an integer $\lambda\geq\max(\lambda_1,\lambda_2)$.\end{proof}
\begin{lemma} \label{lemma LU}
Let $(\beta_n)$ be a slowly oscillating sequence of positive numbers. 
Let $A = (a_{m, n})_{m, n}$ be the matrix of a bounded operator on $\ell_2$. Assume that, for some integer $\lambda>1$, and some constants $c$, $b$, we have:
\smallskip

$1)$ $|a_{m, n}| \leq c \, \e^{- b n}$ \; when $n\geq\lambda m$;
\medskip

$2)$ $|a_{m, n}| \leq c \, \e^{- b m}$ \; when $m \geq \lambda n$.
\smallskip

\noindent Then the matrix $\dis \tilde A = \bigg( a_{m, n} \, \sqrt{\frac{\beta_m}{\beta_n}} \bigg)_{m, n}$ also defines a bounded operator on $\ell_2$. 
\end{lemma}
\begin{proof} In the sequel $\|\cdot\|$ stands for the $\ell^2$-norm.

Since $\beta$ is slowly oscillating, it has polynomial decay and growth: for some $\alpha, \gamma > 0$ and $\delta \in (0, 1)$, we have 
$\delta \, (n+1)^{- \alpha} \leq \beta_n \leq \delta^{- 1} (n+1)^\gamma$.

The matrix $\tilde A$ is Hilbert-Schmidt far from the diagonal since
\begin{align*}
\sum_{n = 1}^\infty \sum_{\lambda m < n} |a_{m, n}|^2 \, \frac{\beta_m}{\beta_n} 
&\lesssim \sum_{n = 1}^\infty \sum_{\lambda m < n} \, (n+1)^{\alpha + \gamma} |a_{m, n}|^2 \\
&\lesssim \sum_{n = 1}^\infty (n+1)^{\alpha + \gamma + 1} \e^{- 2b n} < + \infty \, ,
\end{align*}
and
\begin{align*} 
\sum_{n = 0}^\infty \sum_{m > \lambda n} |a_{m, n}|^2 \, \frac{\beta_m}{\beta_n} 
&\lesssim \sum_{n = 0}^\infty \sum_{m > \lambda n} (n+1)^{\alpha + \gamma} |a_{m, n}|^2 \\ 
&\lesssim \sum_{n = 0}^\infty (n+1)^{\alpha + \gamma} \bigg( \sum_{m > \lambda n} \e^{- 2 b m} \bigg) < + \infty \, . 
\end{align*}
Since $\beta_m / \beta_n$ remains bounded from above and below around the diagonal, the matrix $\tilde A$ behaves like $A$ near the diagonal. 
More precisely, if $I$, $J$ are blocks of integers such that $(m, n) \in I \times J$ implies that $n / \lambda^2 \leq m \leq \lambda^2 n$, then, with obvious notations, 
the slow oscillation of $\beta$ gives, for some $C>0$: 
\begin{align*}
\Bigg| \sum_{(m, n) \in I \times J} a_{m, n} \, x_n \overline{y_m} \, \sqrt{\frac{\beta_m}{\beta_n}} \, \Bigg| 
& \leq \| A \| \, \bigg( \sum_{(m, n) \in I \times J} |x_n|^2 \, |y_m|^2 \, \frac{\beta_m}{\beta_n} \bigg)^{1 /2} \\ 
& \leq C^{1 /2} \| A \| \, \| P_J x \| \, \| P_I y \| \, . 
\end{align*}

For $k =0, 1, 2, \ldots\, ,$ let $J_k =[\lambda^k, \lambda^{k  + 1}[\,$ and, for $k =1,2,\ldots\,$, we define $I_k =[\lambda^{k-1},\lambda^{k+2}[$. We also 
define $I_0 = [0, \lambda^2[$.

We define the matrix $R$, whose entries are
\begin{align*} 
\phantom{+ 1} r_{m, n} & =  \left\{ 
\begin{array}{lcl}
\dis \sqrt{\frac{\beta_m}{\beta_n}} \, a_{m, n} & \text{if} & \dis (m, n) \in \bigcup_{k = 0}^\infty (I_k \times J_k) \\
\quad 0 & \text{elsewhere.}
\end{array}
\right.
\end{align*} 
Let $H_k$ be the subspace of the sequences $(x_n)_{n \geq 0}$ in $\ell_2$ such that $x_n = 0$ for $n \notin I_k$, i.e. 
$H_k = {\rm span}\, \{e_n \tq n \in I_k \}$, and ${\tilde H}_k ={\rm span}\, \{e_n \tq n \in J_k \}$. 
Let $P_k$  be (the matrix of) the orthogonal projection of $\ell_2$ with range $H_k$ and $Q_k$ that with range ${\tilde H}_k$. 
Then $R_k = P_k A Q_k$ is the matrix with entries $a_{m, n}$ when $(m, n) \in I_k \times J_k$ and $0$ elsewhere.
By the above discussion, we have 
\begin{displaymath}
| ( R_k x \mid y ) | \leq C^{1 /2} \| A \| \, \| Q_k x \| \, \| P_k y \| \, .
\end{displaymath}

Point out that, for every $y\in\ell^2$, we have $\dis\sum  \| P_k y \|^2\le3\|y\|^2$ since each integer belongs to at most $3$ intervals $I_k$.

In the same way, for every $x\in\ell^2$, we have $\dis\sum  \|Q_k x \|^2\le\|x\|^2$ since the subspaces $\tilde H_k$ are orthogonal.

Summing up over $k$, we get the boundedness of $R = \! \sum_{k = 0}^\infty R_k$. 
\smallskip

Now let us check when the entries of $R$ do not coincide with the entries of $\tilde A$. Actually it happens when $(m,n)$ does not belong to the union 
of the $I_k \times J_k$. When $n\ge1$, it means that $n$ belongs to some $J_p$ but $m\notin I_p$: either $m<\lambda^{p-1}$ or $m\ge\lambda^{p+2}$, 
hence either $m/n<\lambda^{-1}$ or $m/n>\lambda$.  Therefore the non zero entries $(m,n)$ of $\tilde A - R$ satisfy either $n>\lambda m$ or $m>\lambda n$. 
\smallskip

That ends the proof since we have seen at the beginning that $\tilde A - R$ is Hilbert-Schmidt.
\end{proof}
\noindent {\bf Remark.} The proof shows that, instead of $1)$ and $2)$, it is enough to have:
\begin{displaymath} 
\sum_{m < C_1 n} n^{\alpha + 1} |a_{m, n}|^2 < \infty \quad \text{and} \quad \sum_{m > C_2 n} m^\alpha |a_{m, n}|^2 < \infty \, .
\end{displaymath} 
Moreover the proof also shows that when $\beta$ is slowly oscillating, if we set $E = \{ (m, n) \tq C_1 n \leq m \leq C_2 n\}$, for some $C_1$,$C_2>0$, then the 
matrix $\big( \sqrt{\beta_m / \beta_n} \, \ind_E (m, n) \big)$ is a Schur multiplier over \emph{all} the bounded matrices, while Kacnel'son's theorem 
(Theorem~\ref{theo Kacnelson}) says that, if $\gamma = (\gamma_n)$ is non-increasing, the matrix $(\gamma_m / \gamma_n)$ is a Schur 
multiplier of all bounded \emph{lower-triangular} matrices. 

\begin{proof} [Proof of Theorem \ref{theo CS T_a beta SO}]
Thanks to Lemma~\ref{majo coeff}, the hypotheses of Lemma~\ref{lemma LU} are fulfilled by the matrix whose entries are $a_{m,n} = \widehat{\phi^{n}} (m)$. 
It follows (with the notations of Lemma~\ref{lemma LU}) that $\tilde A$ is bounded on $\ell^2$, which means  exactly that $T_a$ is bounded on $H^2(\beta)$.
\end{proof}

\goodbreak
\subsection {Necessary condition} \label{subsec: CN for T_a} 

The main theorem of this section is the following. 

\begin{theorem} \label{theo CN T_a beta SO}
If the composition operator $C_{T_a}$ is bounded on $H^2 (\beta)$ for some $a \in \D \setminus \{0\}$, then $\beta$ is slowly oscillating. 
\end{theorem}

Let us give right now a corollary of this result. 

\begin{proposition} \label{T_a not bounded}
For the weight $\beta$ constructed in the proof of Theorem~\ref{theo example}, no automorphism $T_a$ with $0 < a < 1$ can be bounded.
\end{proposition}
\begin{proof}
Indeed, it is clear that $\beta$ is not slowly oscillating, since 
\begin{displaymath}
\beta_{(k + 1)! - 1} / \beta_{(k + 1)!} = 1 / (k + 1) \converge_{k \to \infty} 0 \, . \qedhere
\end{displaymath}
\end{proof}

To prove Theorem~\ref{theo CN T_a beta SO}, we need estimates on the Taylor coefficients of $T_a^n$. 
Actually the Taylor coefficients of $T_a^n$ are the Fourier coefficients of $x\in\R\mapsto T_a^n(\e^{ix})$ and we shall denote them with the same notation 
$\widehat{T_a^n}$. Sharp such estimates are given in the papers \cite{SZZA-JAM} and \cite{SZZA}, and we thank R.~Zarouf for interesting informations in 
this respect. Our method, using stationary phase and the van der Corput lemma, is a variant of that used in \cite{SZZA-JAM}, \cite{SZZA}, and goes back at 
least to \cite{Girard}. However, we need minorations of $|\hat{T_a^n} (m)|$ when $m$ is close to $n$, and Szehr and Zarouf's estimates show that this quantity 
oscillates and, for individual $a$, can be too small for our purpose, so we cannot use them, and have to prove an estimate \emph{in mean} for $a$ in some 
subinterval of $(0, 1)$. 

\medskip

We begin with a standard fact, that we give with its proof for the convenience of the reader. 

\begin{lemma} \label{lemma Poisson} 
Let $a \in (0, 1)$ and let 
\begin{displaymath} 
P_{- a} (x) = \frac{1 - a^2}{1 + 2 \, a \cos x + a^2} 
\end{displaymath} 
be the Poisson kernel at the point $- a$. Then, for all $x \in [- \pi, \pi ]$:
\begin{equation} 
T_{a} (\e^{i x}) = \exp \big[ i \, V_a (x) \big] \, ,
\end{equation} 
where 
\begin{equation} 
V_{a} (x) = \int_{0}^x P_{- a} (t) \, dt \, . 
\end{equation} 
\end{lemma}
\begin{proof}
For $t \in [- \pi, \pi]$,  write:
\begin{displaymath} 
\psi (t) := \frac{\e^{it} + a}{1 + a\, \e^{it}} = \exp \, (i \, v (t) ) \, ,
\end{displaymath} 
with $v$ a real-valued, ${\cal C}^1$-function on $[- \pi, \pi]$ such that $v (0) = 0$. This is possible since $|\psi (\e^{it})| = 1$ and $\psi (0) = 1$. 
Differentiating both sides with respect to $t$, we get:
\begin{displaymath} 
i \, \e^{it} \frac{1 - a^2}{(1 + a \, \e^{it})^2} = i \, v ' (t) \, \frac{\e^{it} + a \,}{1 + a \, \e^{it}} \, \cdot
\end{displaymath} 
This implies 
\begin{displaymath} 
v ' (t) = \frac{1 - a^2}{|1 + a \, \e^{it}|^2} = P_{- a} (t) \, ,
\end{displaymath} 
and the result follows since $v (0) = 0 = V_a (0)$. 
\end{proof}

Let us note that, with $V_a$ the function of Lemma~\ref{lemma Poisson}, the Fourier formulas give, since $\widehat{T_{a}^{n}} (m)$ is real, or since 
$n V_a (x) - m x$ is odd: 
\begin{equation} \label{egalite a et I} 
2 \pi\,\widehat{T_{a}^{n}} (m) = \int_{- \pi}^{\pi} \exp(i [n V_a (x) - m x]) \, dx = 2 \, \Re I_{m, n} \, ,
\end{equation} 
where 
\begin{equation} \label{def I_{m, n}}
I_{m, n} = \int_{0}^{\pi} \exp i [n V_a (x) - m x ] \, dx \, .
\end{equation} 

Now the main ingredient for proving Theorem~\ref{theo CN T_a beta SO} is the following. 
\begin{proposition} \label{sept} 
Let   $I := \big[\frac{1}{2} \raise 1 pt \hbox{,} \frac{2}{3} \big]$. There exist constants $\alpha > 1$, e.g. $\alpha = 5/4$, and 
$\delta \in (0, 1/2)$  such that, for $n$ large enough ($n \geq n_0$), it holds
\begin{equation} 
\qquad \qquad \int_{I} |\widehat{T_{a}^{n}} (m)|^2 \, d a \geq \frac{\delta}{n} \qquad \text{for all } m \in [\alpha^{-1} n, \alpha \, n] \, . \qquad{}
\end{equation} 
\end{proposition}
\begin{proof}  We will set once and for all 
\begin{equation}
q = \frac{m}{n} \, \raise 1 pt \hbox{,}
\end{equation}
so that $\alpha^{-  1} \leq q \leq \alpha$ where $\alpha = 5/4$ (say).  We will only consider pairs $(a, q)$ satisfying 
\begin{equation}
a \in I = \bigg[ \frac{1}{2} \, \raise 1.5 pt \hbox{,} \, \frac{2}{3} \bigg] \, , \qquad  
q \in J := \bigg[ \frac{4}{5} \, \raise 1.5 pt \hbox{,} \, \frac{5}{4} \bigg] \, \cdot
\end{equation}
Such pairs will be called \emph{admissible}. 
\smallskip

With these notations, we set, for $0 \leq x \leq \pi$:
\begin{equation} \label{fonction} 
F_{q}(x)=V_a(x)-\frac{m}{n}x= \int_{0}^x P_{- a} (t) \, dt - q x  \, , 
\end{equation}
where $P_{-a}$ is the Poisson kernel at $- a$. We have
\begin{displaymath}
F'_{q} (x) = \frac{(1 - a^2)}{1 + 2 a \cos x + a^2} - q \, . 
\end{displaymath}
and the unique (if some) critical point $x_q = x_{q} (a)$ of $F_q$ in $[0, \pi]$ is given by $P_{- a} (x_q) = q$, that is 
\begin{equation}\label{critic} 
\cos x_q = \frac{1}{q} \, \frac{1 - a^2}{2 a} - \frac{1 + a^2}{2 a} =: h_{q}(a) \, .
\end{equation}
We now proceed through a series of simple lemmas, and begin by estimates on $h_q$ and $ x_q$.

\begin{lemma}\label{un} 
There are positive constants $C > 1$ and $\delta \in (0, 1 / 2)$ such that, for every admissible couple $(a, q)$, we have:
\begin{equation}
\qquad |h_{q} (a) |\leq 1 - \delta \quad  \text{and} \quad  |h'_{q} (a) | \leq C \, ,
\end{equation}
so, there is one critical point $x_{q} (a)$ satisfying
\begin{equation}\label{encx}
\qquad \delta \leq x_{q} (a) \leq \pi - \delta \quad \text{and} \quad \sin x_{q}(a) \geq \delta  \, ;
\end{equation}
moreover
\begin{equation}
\qquad | x'_{q} (a) | \leq C \quad \text{and} \quad \delta \leq |P'_{- a} (x_q) | \leq C \, .
\end{equation}
\end{lemma}
\begin{proof}
We have 
\begin{displaymath}
h_{q} (a) = \bigg( \frac{1}{q} \, \frac{1 - a^2}{2 a} \bigg) + \bigg( - \frac{1 + a^2}{2 a} \bigg) =: u(a) + v(a) \, ,
\end{displaymath}
with $u$ and $v$ respectively decreasing and increasing on $[0, 1]$, and with $v \leq 0$, so that we have, for $q \in J$:
\begin{displaymath}
h_{q} (a) \leq u \Big(\frac{1}{2} \Big) = \frac{3}{4 q} \leq \frac{15}{16} \, \cdot
\end{displaymath}
Similarly:
\begin{displaymath}
h_{q} (a) \geq u \Big( \frac{2}{3} \Big) + v \Big( \frac{1}{2} \Big) = \frac{5}{12 q} - \frac{5}{4} \geq \frac{1}{3} - \frac{5}{4} = - \frac{11}{12} \, \cdot
\end{displaymath}
Next, $2 h'_{q} (a)  = \big(1 - \frac{1}{q} \big) \, \frac{1 \,}{a^2} - \big( 1 + \frac{1}{q} \big) \,$; hence $| h'_{q} (a) | \leq C$. 
So that, writing $x_q = x_{q} (a) = \arccos h_{q} (a)$, we get, with another constant $C>0$, 
\begin{displaymath}
| x'_{q} (a) | = \frac{ | h'_{q} (a) | }{\sqrt{1 - h_{q} (a)^2}} \leq C \, ,
\end{displaymath}
since $h_{q} (a)^2 \leq 1 - \delta$. Finally, $1/9 \leq (1 - a)^2 \leq 1 + 2 a \cos x_q + a^2 \leq 4$ and since
\begin{displaymath}
P'_{- a} (x_q) = \frac{2 a (1 - a^2) \sin x_q}{(1 + 2 a \cos x_q + a^2)^2} \, \raise 1 pt \hbox{,}
\end{displaymath}
we get the final estimates, ending the proof.
\end{proof}

Back to Proposition~\ref{sept}. 
\smallskip

We saw in \eqref{egalite a et I} that the value of $a_{m, n} := \widehat{T_{a}^{n}} (m)$ is given by the formula 
\begin{equation} \label{egalite a et I bis}
a_{m, n} = \frac{1}{\pi} \,\Re I_{m, n} \, . 
\end{equation}

We have the following estimate, whose proof is postponed
(recall that $q = m / n$, and $x_q = x_{q}(a)$). 

\begin{proposition} \label{coro stationary phase}
We have  
\begin{equation} \label{cruc} 
I_{m, n} = \sqrt{2 \pi} \, n^{- 1/2} \,\, \frac{\e^{i \, [nF_q (x_q) + \pi/4]}}{\sqrt{|P'_{ -a} (x_q)|}} + O \, (n^{- 3/5}) \, ,
\end{equation} 
where the $O$ only depends on $a$, and so is absolute as long as $(a, q)$ is admissible. 
\end{proposition}

Note that  $3/5 > 1/2$. We hence have:
\begin{equation}
a_{m,n} = \sqrt{ \frac{2}{\pi} } \, n^{- 1/2} \,\, \frac{\cos \, [ \pi / 4 + n F_{q} (x_q)]}{\sqrt{|P_{- a} ' (x_q)|}} + O \, (n^{- 3/5}) \, ,
\end{equation}
It will be convenient to introduce $\phi_q (a)$, by settting  
\begin{equation}\label{nota} 
F_{q} \big( x_{q} (a) \big) = \varphi_{q} (a) \, . 
\end{equation}
Then, since $1/2 + 3/5 = 11/10$, $\cos^{2}(\pi / 4 + x) = \frac{1 - \sin 2x}{2}$ and  $ |P'_{- a} (x_q) |\geq \delta$  by Lemma~\ref{un}:
\begin{displaymath}
a_{m, n}^{\, 2} = \frac{1}{\pi} \, n^{- 1} \, \frac{1 - \sin \, [2 n \, \varphi_{q} (a) ] }{|P_{- a}' (x_q)|} + O\, (n^{- 11 / 10}) \, ,
\end{displaymath}
implying, since $ |P'_{- a} (x_q) |\leq C$ by Lemma~\ref{un} (again for $(a, q)$ admissible), and changing $\delta$: 
\begin{equation}\label{basic} 
a_{m, n}^{\, 2} \geq \delta \, n^{- 1} (1 - \sin \, [2 n \, \varphi_{q} (a)] ) + O\, (n^{- 11 / 10}) \, .
\end{equation}

We will also need estimates on the derivatives of $\varphi_{q}(a)$.

 \begin{lemma}\label{deux}  
If $(a, q)$ is admissible, then $\varphi_q$ decreases on $I$ and moreover:
\medskip

$1)$ $ \quad | \varphi '_{q} (a) | \geq \delta$; 
\medskip

$2)$ \quad $| \varphi ''_{q} (a) | \leq C$. 
\end{lemma}
\begin{proof} 
Note, in passing, that, with $x = x_{q} (a) \in [0, \pi]$ (thanks to \eqref{fonction}):
\begin{displaymath}
\varphi_{q} (a) = \int_{0}^x [ P_{- a}  (t) - P_{- a} (x) ] \, dt \leq 0 \, ,
\end{displaymath}
since the integrand is negative. Next, if $f$ and $g$ are real ${\cal C}^1$-functions and 
\begin{displaymath}
\Phi (a) = \int_{0}^{f (a)} g (a, t) \, dt \, ,
\end{displaymath}
the chain rule gives 
\begin{displaymath}
\Phi ' (a) = f  ' (a) \, g \big(a, f (a) \big) + \int_{0}^{f (a)} \frac{\partial g}{\partial a} \, (a, t) \, dt \, .
\end{displaymath}
With $g (a, t) = P_{- a} (t)$ and $f (a) = x_{q} (a)$, we get, remembering that $x_{q} (a)$ is critical for $F_q$: 

\begin{align*}
\varphi '_{q} (a) &= \big[ P_{- a} \big( x_{q}  (a)  \big) - q \big] \, x_{q} ' (a) 
+ \int_{0}^{x_{q} (a)} \frac{\partial P_{- a}}{\partial a} \, (a, t) \, dt\\ 
&= \int_{0}^{x_{q} (a)} \frac{\partial P_{- a}}{\partial a} \, (a, t) \, dt \, . 
\end{align*}

But $P_{- a} (t) = 1 + 2 \sum_{k = 1}^\infty (- a)^k \cos k t$, so that

\begin{displaymath}
\varphi '_{q} (a) = \int_{0}^{x_{q} (a)} \bigg( - 2 \sum_{k = 1}^\infty  k \, (- a)^{k - 1} \cos k t \bigg) \, dt 
= \frac{2}{a} \, \sum_{k = 0}^\infty (- a)^k \sin \, [ k x_{q} (a) ] \, ,
\end{displaymath}
that is
\begin{equation} \label{signe} 
\varphi '_{q} (a) = \frac{2}{a} \, \Im \frac{1}{1 + a \, \e^{i  x_{q} (a)}} = \frac{- 2 \sin x_{q} (a)}{1 + 2 a \cos x_{q} (a) + a^2}< 0 \, .
\end{equation}

Now, \eqref{encx} gives $1)$.
\smallskip

Since $|x'_{q} (a) | \leq C$ by Lemma~\ref{un}, the chain rule and \eqref{signe} clearly give the uniform boundedness of 
$|\varphi ''_{q} (a)|$ when $(a, q)$ is admissible, and this ends the proof.
\end{proof}

Lemma~\ref{un} and Lemma~\ref{deux} will now be exploited through a simple variant of  the van der Corput inequalities. 

\begin{lemma} \label{vdc} 
Let $f \colon [A, B] \to \R$, with $A < B$, be a ${\cal C}^2$-function satisfying  $|f ' | \geq \delta$ and  $ |  f '' | \leq C $, and let us put
$M = \int_{A}^B \e^{i n  f (x)} \, dx$. Then
\begin{displaymath}
| M | \leq \frac{2}{n \delta} + \frac{C (B - A)}{n \, \delta^2} \, \cdot 
\end{displaymath}
\end{lemma}
\begin{proof} 
Write 
\begin{displaymath}
\e^{i n  f} = \frac{(\e^{i n f})'}{i n  f '}
\end{displaymath}
and integrate by parts to get
\begin{displaymath}
M = \bigg[ \frac{\e^{i n f}}{i nf '} \bigg]_{A}^{B} - \frac{i}{n} \int_{A}^B \e^{i n f (x)} \frac{f '' (x)}{[f ' (x)]^2} \, dx 
=: M_1 + M_2 \, , 
\end{displaymath}
with $| M_1 | \leq \frac{2}{n  \delta}$ and $| M_2 | \leq \frac{B - A}{n} \, \frac{C}{\delta^2}$. This gives the result.
\end{proof}
\goodbreak 

\emph{End of proof of Proposition~\ref{sept}}. The preceding lemma can be applied with $A = 1/2$, $B = 2/3$, $f = \varphi_q$ and $n$ changed into $2 n$, 
since Lemma~\ref{deux} shows that this $f$ meets the assumptions of Lemma~\ref{vdc}. This gives us, uniformly with respect to $(a,q)$ admissible:
\begin{equation} \label{fin} 
\bigg| \int_{I} \sin \, [2 n \varphi_{q} (a)] \, d a \bigg| \leq \bigg| \int_{I} \e^{2 i n \varphi_{q}(a)} \, d a \bigg| \leq \frac{C}{n} \, \cdot
\end{equation} 

Now, integrating \eqref{basic} on $I$ and using \eqref{fin} give for some numerical $\delta \in (0, 1/2)$  (recall that 
$a_{m, n} = \widehat{T_{a}^{n}} (m)$): 
\begin{displaymath}
\int_{I} | \widehat{T_{a}^{n}} (m)|^2 \, d a \geq \delta n^{- 1} + O\, (n^{- 2}) + O\, (n^{- 11 / 10}) \geq ( \delta/2) \, n^{- 1}
\end{displaymath}
for $n \geq n_0$ and $\alpha^{- 1} \leq m / n \leq \alpha$. This ends the proof of Proposition~\ref{sept}.
\end{proof}
%

%
\begin{proof}[Proof of Theorem~\ref{theo CN T_a beta SO}] 
We know, by Proposition~\ref{tous pour un}, that $C_{T_a}$ is bounded for all $a \in \D$, and that, thanks to Lemma~\ref{second lemma}, 
\begin{displaymath}
K := \sup_{1/ 2 \leq a \leq 2 / 3} \| C_{T_a} \| < + \infty \, . 
\end{displaymath}
Matricially, this writes, for all $a \in ( 1 / 2, 2 / 3)$:
\begin{displaymath}
\bigg\| \bigg( \hat{T_a^n} (m) \, \sqrt{ \frac{\beta_m}{\beta_n} } \bigg)_{m, n} \bigg\| \leq K \, . 
\end{displaymath}
In particular, for every $n \geq 1$, we have, considering the columns and rows of the previous matrix:
\begin{displaymath}
\sum_{m = 1}^\infty | \hat{T_a^n} (m) |^2 \, \frac{\beta_m}{\beta_n}\leq K^2 \, , \quad  \text{i. e.} \quad 
\sum_{m = 1}^\infty | \hat{T_a^n} (m) |^2 \, \beta_m \leq K^2 \, \beta_n \, , 
\end{displaymath}
and, for every $m \geq 1$:
\begin{displaymath}
\sum_{n = 1}^\infty | \hat{T_a^n} (m) |^2 \,  \frac{\beta_m}{\beta_n} \leq K^2 \, , \quad \text{i. e.} \quad 
\sum_{n = 1}^\infty | \hat{T_a^n} (m) |^2 \, \frac{1}{\beta_n} \leq \frac{K^2}{\beta_m} \, \cdot 
\end{displaymath}
In particular, for every $n \geq 1$:
\begin{equation} \label{eq: line}
\sum_{(4 / 5) n \leq j \leq (5 / 4) n} | \hat{T_a^n} (j) |^2 \, \beta_j \leq K^2 \, \beta_n \, \cdot 
\end{equation}
and, for every $m \geq 1$: 
\begin{equation} \label{eq: column}
\sum_{(4 / 5) m \leq k \leq (5 / 4) m} | \hat{T_a^k} (m) |^2 \, \frac{1}{\beta_k} \leq \frac{K^2}{\beta_m} \, \cdot 
\end{equation}
Integrating on $a \in (1 / 2, 2 / 3)$, and using Proposition~\ref{sept}, we get, from \eqref{eq: line}, for $n$ large enough,
\begin{equation} \label{first piece} 
\frac{\delta}{n} \sum_{(4 / 5) n \leq j \leq (5 / 4) n} \beta_j \leq \frac{K^2}{6 \,} \, \beta_n
\end{equation}
and, from \eqref{eq: column}, for $m$ large enough, we have both 

\begin{equation} \label{harmonic mean1}
\frac{\delta}{m} \sum_{(4 / 5) m \leq k\leq  m} \frac{1}{\beta_k} \leq \frac{5 K^2}{24} \, \frac{1}{\beta_m} \, \cdot
\end{equation}
and\begin{equation} \label{harmonic mean2}
\frac{\delta}{m} \sum_{ m \leq k \leq (5 / 4) m} \frac{1}{\beta_k} \leq\frac{5 K^2}{24}\, \frac{1}{\beta_m} \, \cdot
\end{equation}
Since the harmonic mean (over the sets of integers $[4m/5, m ]$ and $[m, 5 m/4 ]$, which have cardinality $\approx n\approx m$) is less than the arithmetical mean,  
we obtain from \eqref{harmonic mean1} and \eqref{harmonic mean2} both
\begin{equation} \label{second piece1}
\beta_m \leq \frac{125}{24\delta} \, \frac{K^2}{m} \sum_{(4 / 5)m \leq k \leq m} \beta_k \, . 
\end{equation}
and
\begin{equation} \label{second piece2}
\beta_m \leq \frac{10}{3\delta} \, \frac{K^2}{m} \sum_{m \leq k \leq (5 / 4) m} \beta_k \, . 
\end{equation}

Now assume that $n \leq m \leq (5 / 4) n$. From \eqref{second piece1}, we have 
$$\beta_m \lesssim \, \frac{1}{m} \sum_{(4 / 5)m \leq k \leq m} \beta_k \lesssim \frac{1}{n} \sum_{(4 / 5)n \leq k \leq (5/4)n} \beta_k \lesssim \beta_n  $$ 
thanks to \eqref{first piece}.\\
From \eqref{second piece2} and  \eqref{first piece}, we treat the case $(4/5)n \leq m \leq  n$ in the same way.\\
We conclude that for some constant $c>0$, we have for $n$ and $m$ large enough satisfying $(4 / 5) n \leq m \leq (5 / 4) n$:
\begin{equation}
\beta_m \leq c \, \beta_n \, , 
\end{equation}
which means that $\beta$ is slowly oscillating. 
\end{proof}
%

%
\begin{proof} [Proof of Proposition~\ref{coro stationary phase}] 
We will use a variant of \cite[Lemma~4.6, p.~72]{TIT} on the van der Corput's version of the stationary phase method. 
A careful reading of the proof in \cite[page~72]{TIT} gives the version below, which only needs local estimates on the second derivative $F''$, as 
occurs in our situation. For sake of completeness, we will give a proof, however postponed to the Appendix.
\goodbreak

\begin{proposition} [Stationary phase] \label{tihebr}  
Let $F$ be a real function with continuous derivatives up to the third order on the interval $[ A, B]$ and $F'' > 0$ throughout $]A, B[\,$.  Assume that there is a 
(unique) point $c$ in $]A, B[$ such that $F' (c) = 0$, and that, for some positive numbers $\lambda_2$, $\lambda_3$, and $\eta$, the following assertions hold: 
\begin{enumerate}
\setlength\itemsep {-0.1 em}

\item [$1)$] $[c - \eta, c + \eta] \subseteq [A,B]$;

\item [$2)$] $F '' (x) \geq \lambda_2$ for all $x \in [c - \eta, c + \eta]$;

\item [$3)$]  $|F ''' (x)| \leq  \lambda_3$  for all $x \in [A, B]$.
 \end{enumerate}
Then:
\begin{equation} \label{ozouf} 
\int_A^B \e^{i F(x)} \, dx = \sqrt{2 \pi} \, \, \frac{\, \e^{i [F (c) + \pi/4]} \, }{F '' (c)^{1/2}} + O \, \bigg( \frac{1}{\eta \lambda_2} 
+ \eta^4 \lambda_3 \bigg) \, ,
\end{equation}
where the $O$ involves an absolute constant. 
\end{proposition}

We will show that Proposition~\ref{tihebr} is applicable with $F=nF_q$ and
\begin{displaymath} 
[A, B] = [0, \pi] \, , \quad  c = x_q \, , \quad \lambda_2 = \kappa_0 \, n \, , \quad  \lambda_3 = C_0 n \, , \quad  
\eta = (\lambda_{2} \lambda_{3})^{- 1/5} \, .
\end{displaymath} 
The parameter $\eta$ is chosen in order to make both error terms in Proposition~\ref{tihebr} equal:  $ \frac{1}{\eta \lambda_2} = \eta^4 \lambda_3$; so:
\begin{displaymath} 
\eta = \kappa \, n^{- 2/5}   
\end{displaymath} 
and 
\begin{equation} \label{big O}
\frac{1}{\eta \lambda_2} + \eta^4 \lambda_3 = \tilde \kappa \, n^{- 3/5} = O\, (n^{- 3/5}) 
\end{equation} 
(with $\kappa = (\kappa_0 C_0)^{- 1 / 5}$ and $\tilde \kappa = 2 / \kappa_0 \kappa$). 
\smallskip

The slight technical difficulty encountered here is that $F_q ''(x)$ vanishes at $0$ and $\pi$. But Proposition~\ref{tihebr} covers this case. We have 
\begin{displaymath} 
F''(x)=nF_q ''(x) = n P'_{- a} (x) = 2 \, a (1 - a^2) \, \frac{\sin x}{(1 + 2 a \cos x + a^2)^2} \, n \, ,
\end{displaymath} 
and there are some positive (and absolute) constants $\kappa_0$ and $\sigma$ such that 
\begin{equation} \label{rectif} 
F''(x) \geq \kappa_0 \, n = \lambda_2  \quad \text{for } x \in [\sigma , \pi - \sigma] \, .
\end{equation}

Now (for $n$ large enough), $[x_q - \eta, x_q + \eta] \subseteq  [\sigma, \pi - \sigma]$. Hence the assumptions $1)$ and $2)$ of 
Proposition~\ref{tihebr} are satisfied. 
\smallskip

Finally, since $F(x)=nF_q (x) = n [V_a (x) - q x]$, $F'''=nF_q ''' =nV_a '''=nP_{-a} ''$, we have, for all $x \in [0, \pi]$ and $(a,q)$ admissible:
\begin{displaymath} 
|F ''' (x) | \leq  C_0 \, n = \lambda_3 \, ,
\end{displaymath} 
where $C_0$ is absolute
and assertion $3)$ of Proposition~\ref{tihebr} holds. 
\smallskip

With \eqref{big O} this ends the proof of \eqref{cruc}, once we remarked that $n V_a '' (x_q) = F '' (x_q)$. 
\end{proof}

\goodbreak 
\section {Boundedness of all composition operators} \label{all symbols}

In this section, we characterize all the sequences $\beta$ for which all composition operators are bounded on $H^2 (\beta)$. The main remaining step is 
the following theorem.

\begin{theorem} \label{example without ED} 
\hskip - 2,2 pt Assume that all composition operators $C_\phi$ are bounded on $H^2 (\beta)$. Then $\beta$ is essentially 
decreasing.
\end{theorem}

As an immediate consequence, we obtain Theorem~\ref{full characterization}.

\begin{proof} [Proof of Theorem~\ref{full characterization}] 
Assume that $\beta$ is essentially decreasing and slowly oscillating. All composition operators 
$C_\psi$ with $\psi (0) = 0$ are bounded on $H^2(\beta)$ (see the Introduction or Proposition~\ref{phi(0)=0}). Since $\beta$ is slowly oscillating, all the 
composition operators $C_{T_a}$, with $a \in \D$, are bounded thanks to Theorem~\ref{theo CS T_a beta SO}. Now it is very classical that we can get the 
boundedness of every composition operators. Indeed given a symbol $\varphi$, the symbol $\psi=T_a\circ\varphi$ fixes the origin for $a=-\varphi(0)$. 
Since $C_\varphi = C_\psi \circ C_{T_{-a}}$, the conclusion follows. 
\smallskip

Assume that all composition operators are bounded on $H^2 (\beta)$, in particular, the $C_{T_a}$ ones are so, and $\beta$ is slowly oscillating, thanks to 
Theorem~\ref{theo CN T_a beta SO}. It also follows from Theorem~\ref{example without ED} that $\beta$ is essentially decreasing. 
\end{proof}

We will use the following elementary, but crucial, lemma.

\begin{lemma} \label{crucial lemma}
Let $u$ be a function analytic in an open neighborhood $\Omega$ of $\overline{\D}$. Then, for every $\eps > 0$, there exists an integer $N \geq 1$ such that 
\begin{equation}
\sum_{j = N p}^\infty |\hat{u^p} (j) |^2 \leq \eps \, , \quad \forall p \geq 1 \, .  
\end{equation}
\end{lemma}
\begin{proof}
From Lemma~\ref{majo coeff}, we know that there exist some integer $\lambda>1$ and a constant $b > 0$, such that $|\hat{u^p} (j) | \leq
\e^{- b j}$ when $j \geq\lambda p$. Therefore, for any $N\geq\lambda$, we have
\begin{displaymath}
\sum_{j = Np}^\infty | \hat{u^p} (j) |^2 \leq\big(1-\e^{-2b}\big)^{-1}\e^{-2bNp}\leq\big(1-\e^{-2b}\big)^{-1}\e^{-2bN}\leq\eps \, , 
\end{displaymath}
as soon as $N$ is chosen large enough.
\end{proof}
\begin{proof} [Proof of Theorem~\ref{example without ED}] 
Thanks to Theorem~\ref{theo CN T_a beta SO}, we know that $\beta$ is slowly oscillating.
Therefore we shall assume that the sequence $\beta$ is  slowly oscillating but is  not essentially decreasing. 
We are going to construct an analytic function $\phi \colon \D \to \D$ such that the composition operator $C_\phi$ is not bounded on $H^2 (\beta)$. 
This function $\phi$ will be a Blaschke product, of the form 
\begin{displaymath}
\phi (z) = \prod_{k = 1}^\infty T_{a_k} (z^{n_k}) = \prod_{k = 1}^\infty \frac{\! z^{n_k} + a_k}{\, 1 + a_k z^{n_k}} \, \raise 1 pt \hbox{,}
\end{displaymath}
for a sequence of numbers $a_k \in (0, 1)$ such that $\sum_{k\geq 1}(1-a_k)<\infty$ and a sequence of positive integers $n_k$ increasing to infinity. 
We recall that $T_a (z) = \frac{z + a}{1 + a z}$ for $0 < a < 1$. Then, $$|T_{a_k}(z^{n_k})-1|\leq \frac{2(1-a_k)}{1-|z|}$$ and $\varphi$ is a locally uniformly 
convergent infinite product. Observe that $\phi$ is indeed a convergent Blaschke product, with $n_k$ zeroes of modulus $a_{k}^{1/n_k},\  k=1,2,\ldots, $ and, 
setting $a_k=\e^{-\eps_k}$,  $\sum_{k}n_k(1-a_{k}^{1/n_k})\leq \sum_{k} n_k (\eps_k/n_k)=\sum_{k} \eps_k<\infty$.
\smallskip

These sequences  will be constructed by induction, together with another sequence of integers $(m_k)_{k \geq 1}$. 
\smallskip

Since $\beta$ is not essentially decreasing, there exist integers $n_1 > m_1\geq 4$ such that $\beta_{n_1} \geq 2 \, \beta_{m_1}$. 
We start with
\begin{displaymath}
a_1 = 1 - \frac{1}{m_1} \geq \frac{3}{4} \, \cdot
\end{displaymath}

Using Lemma~\ref{crucial lemma} with $u = T_{a_1}$, we get $N_0 \geq 1$ such that 
\begin{displaymath}
\qquad \sum_{j = N_0 m}^\infty | \hat{T_{a_1}^m} (j) |^2 \leq 2^{- 15} \, , \quad \forall m \geq 1 \, .
\end{displaymath}

Assume now that we have constructed increasing sequences of integers 
\begin{displaymath}
m_1, m_2, \ldots, m_k, \quad n_1, n_2, \ldots, n_k, \quad N_0, N_1,  \ldots, N_{k - 1} 
\end{displaymath}
such that, for $1 \leq l \leq k - 1$, we have 
\begin{displaymath}
m_{l + 1} \geq 4 m_l \quad \text{and} \quad n_{l + 1} \geq 4 n_l 
\end{displaymath}
and, for $1 \leq l \leq k$:
\begin{displaymath}
n_l \geq N_{l - 1} m_l \quad \text{and} \quad \beta_{n_l} \geq 2^{l} \beta_{m_l}
\end{displaymath}
and 
\begin{displaymath}
\sum_{j = N_{l - 1} m_l}^\infty | \hat{\phi_l^m} (j) |^2 \leq 2^{- 15} \, ,
\end{displaymath}
where 
\begin{displaymath}
a_l = 1 - \frac{1}{m_l} \quad \text{and} \quad \phi_l (z) = T_{a_l} (z^{n_l}) \, . 
\end{displaymath}

We then apply Lemma~\ref{crucial lemma} again, to the function $u = u_k = \phi_1 \cdots \phi_k$. We get $N_k > N_{k - 1}$ such that
\begin{equation} \label{two stars}
\quad \sum_{j = N_k m}^\infty | \hat{u_k^m} (j) |^2 \leq 2^{- 15} \, , \quad \forall m \geq 1 \, .
\end{equation}

Since $\beta$ is not essentially decreasing but is slowly oscillating, there exist $m_{k + 1} \geq 4 m_k$ and $n_{k + 1} \geq 4 n_k$ such that
\begin{displaymath}
n_{k + 1} \geq N_k m_{k + 1} \quad \text{and} \quad \beta_{n_{k + 1}} \geq2^{k+1} \beta_{m_{k + 1}} \, . 
\end{displaymath}
We set 
\begin{displaymath}
a_{k + 1} = 1 - \frac{1}{m_{k + 1}} \quad \text{and} \quad \phi_{k + 1} (z) = T_{a_{k + 1}} (z^{n_{k + 1}}) \, . 
\end{displaymath}
That ends the induction. 
\smallskip

Now, since $\sum_{k = 1}^\infty (1 - a_k) = \sum_{k = 1}^\infty \frac{1}{m_k} \leq \sum_{k = 1}^\infty 4^{- k} = 1 / 3 < + \infty$, the 
infinite product $\phi = \prod_{k \geq 1} \phi_k$ converges uniformly on compact subsets of $\D$. 
\smallskip

To show that the composition operator $C_\phi$ is not bounded on $H^2 (\beta)$, it suffices to show that, for some constant $c_1 > 0$, we have, 
for all $k \geq 2$:
\begin{equation} \label{mino of the sums}
\sum_{j = n_k}^{2 n_k} | \hat{\phi^{m_k}} (j) |^2 \geq c_1 \, . 
\end{equation}
Indeed, since $\beta$ is slowly oscillating, there is a positive constant $\delta < 1$ such that 
\begin{displaymath}
\qquad \beta_j \geq \delta \, \beta_{n_k} \quad \text{for} \quad j = n_k, n_k + 1, \ldots, 2 n_k \, .
\end{displaymath}
Then, if we set $e_k (z) = z^{m_k}$, we have, since $C_\phi (e_k) = \phi^{m_k}$:
\begin{displaymath}
\frac{ \| C_\phi (e_k) \|_{H^2 (\beta)}^2}{\| e_k \|_{H^2 (\beta)}^2} \geq 
\frac{\sum_{j = n_k}^{2 n_k} |\hat{\phi^{m_k}} (j) |^2 \beta_j}{\beta_{m_k}} 
\geq \frac{c_1 \delta \, \beta_{n_k}}{\beta_{m_k}} \geq2^k c_1 \delta \converge_{k \to \infty} + \infty \, , 
\end{displaymath}
so that $C_\phi$ is not bounded on $H^2 (\beta)$. 
\medskip

We now have to show \eqref{mino of the sums}. Let us agree to write formally, for an analytic function $f(z)=\sum_{k=0}^\infty f_k z^k$ and an arbitrary 
positive integer  $p$: 
$$ f (z) = \sum_{k = 0}^p f_k z^k + O (z^{p + 1}).$$ 
For that, we set 
\begin{displaymath}
G_k (z) = \prod_{l = k + 1}^\infty \phi_l (z) = \prod_{l = k + 1}^\infty a_l + O\, (z^{n_{k + 1}}) \, . 
\end{displaymath}
We have, for $k \geq 2$:
\begin{displaymath}
\phi (z) = v_k (z) \, \phi_k (z) \, G_k (z) \, ,
\end{displaymath}
where $v_k = \phi_1 \cdots \phi_{k - 1}$.

Remark now that, for $0 < a < 1$, we have
\begin{displaymath}
T_a (z) = a + (1 - a^2) z + O\, (z^2) \, , 
\end{displaymath}
so
\begin{displaymath} 
\phi_k (z) = T_{a_k} (z^{n_k}) = a_k + (1 - a_k^2) z^{n_k} + O\, (z^{2 n_k}) \, . 
\end{displaymath}
Then
\begin{equation} \label{Taylor G} 
[G_k (z)]^{m_k} = \bigg( \prod_{l = k + 1}^\infty a_l \bigg)^{m_k} + O\, (z^{n_{k + 1}})
\end{equation}
and 
\begin{equation} \label{Taylor phi}
[\phi_k (z)]^{m_k} = a_k^{m_k} + (1 - a_k^2) m_k a_k^{m_k - 1} z^{n_k} + O\, (z^{2 n_k}) \, . 
\end{equation}
But
\begin{equation} \label{mino one} 
a_k^{m_k-1} = \Big( 1 - \frac{1}{m_k} \Big)^{m_k-1} \geq \e^{- 1} := c_2
\end{equation}
and
\begin{equation} \label{mino two} 
(1 - a_k^2) m_k a_k^{m_k - 1} \geq (1 - a_k) m_k a_k^{m_k - 1} \geq c_2 \, .
\end{equation}
\smallskip
Moreover, since $1 - x \geq \e^{- 2 x}$ for $0 \leq x \leq 1 / 2$, we have: 
\begin{align} \label{mino three} 
\bigg( \prod_{l = k + 1}^\infty a_l \bigg)^{m_k}
& \geq \exp \, \bigg( - 2 \Big( \sum_{l = k + 1}^\infty \frac{1}{m_l} \Big) m_k \bigg)  \\ 
& \geq \exp \, \bigg( - 2 \sum_{l = 1}^\infty 4^{- l} \bigg) = \exp ( - 2 / 3) := c_3 \, . \notag
\end{align}

Afterwards, by \eqref{two stars}, we have
\begin{equation} \label{three stars} 
\sum_{j = N_{k - 1} m_k}^\infty | \hat{v_k^{m_k}} (j) |^2 \leq 2^{- 15} \, . 
\end{equation}
Set $v_k^{m_k} = g_1 + g_2$, with
\begin{displaymath}
\left\{
\begin{array}{rl}
g_1 (z) & = \dis \ \sum_{j = 0}^{N_{k - 1} m_k} \hat{v_k^{m_k}} (j) \, z^j \smallskip \\ 
g_2 (z) & = \dis \sum_{j > N_{k - 1} m_k} \hat{v_k^{m_k}} (j) \, z^j \, . 
\end{array}
\right.
\end{displaymath}
By \eqref{three stars}, we have, with $\| \cdot \|_2 = \| \cdot \|_{L^2 (\T)}$:
\begin{displaymath}
\| g_2 \|_2^2 = \sum_{j > N_{k - 1} m_k} | \hat{v_k^{m_k}} (j) |^2 \leq 2^{- 15} \, . 
\end{displaymath}
Besides, since $\phi_k$ is inner as a product of inner functions, we have $| v_k (z) | = 1$ for all $z \in \T$, so 
\begin{displaymath}
\| g_1 \|_2^2 = \| v_k \|_2^2 - \| g_2 \|_2^2 \geq 1 - 2^{- 15} \, . 
\end{displaymath}
Now, $\phi^{m_k} = v_k^{m_k} \phi_k^{m_k} G_k^{m_k} = F_1 + F_2$, with 
\begin{displaymath}
F_1 = g_ 1 \phi_k^{m_k} G_k^{m_k} \quad \text{and} \quad F_2 = g_2 \phi_k^{m_k} G_k^{m_k} \, . 
\end{displaymath}
Using \eqref{Taylor G}, \eqref{Taylor phi}, \eqref{mino two} and \eqref{mino three}, we get
\begin{align*}
\sum_{j = n_k}^{n_k + N_{k - 1} m_k} | \hat{F_1} (j) |^2 
& = \bigg( \prod_{l = k + 1}^\infty a_l \bigg)^{2 m_k} \big[ (1 - a_k^2) m_k a_k^{m_k - 1} \big]^2 \sum_{j = 0}^{N_{k - 1} m_k} | \hat{g_1} (j) |^2 \\ 
& \geq (1 - 2^{- 15}) \, c_2^2 c_3^2 \, . 
\end{align*}
As 
\begin{displaymath}
\| F_2 \|_2^2 \leq \| g_2 \|_2^2 \, \| \phi_k^{m_k} \|_\infty^2 \, \| G_k^{m_k} \|_\infty^2 \leq 2^{- 15} \, , 
\end{displaymath}
we get, using the inequality $| a + b |^2 \geq \frac{1}{2} \, |a |^2 - |b|^2$:
\begin{align*}
\sum_{j = n_k}^{2 n_k} | \hat{\phi^{m_k}} (j) |^2 
& \geq \sum_{j = n_k}^{n_k + N_{k - 1} m_k} | \hat{F_1} (j) + \hat{F_2} (j) |^2 \\ 
& \geq \frac{1}{2} \sum_{j = n_k}^{n_k + N_{k - 1} m_k} | \hat{F_1} (j) |^2 \ - \sum_{j = n_k}^{n_k + N_{k - 1} m_k} | \hat{F_2} (j) |^2 \\ 
& \geq \frac{1}{2} \, (1 - 2^{- 15}) \, c_2^2 \, c_3^2 - 2^{- 15} = \frac{1}{2} \, (1 - 2^{- 15}) \, \e^{- 10 / 3} - 2^{- 15} \\ 
& \geq2^{- 9} - 2^{- 15} > 0 \, . 
\end{align*}

The proof is now complete. 
\end{proof}

\goodbreak
\section {Some results on multipliers} \label{sec: multipliers} 

In this section, we give some results on the multipliers on $H^2 (\beta)$, which show how the different notions of regularity for $\beta$ intervene.
\medskip

The set $\mathcal{M} \big( H^2 (\beta) \big)$ of multipliers of $H^2 (\beta)$ is by definition the vector space of functions $h$ analytic on $\D$ and such that 
$h f \in H^2 (\beta)$ for all $f \in H^2 (\beta)$. When $h \in \mathcal{M} \big( H^2 (\beta) \big)$, the operator $M_h$ of multiplication by $h$ is bounded on 
$H^2 (\beta)$ by the closed graph theorem. The space $\mathcal{M} \big( H^2 (\beta) \big)$ equipped with the operator norm is a Banach space. We note the 
obvious property:
\begin{equation} \label{quinque} 
\mathcal{M} \big( H^2 (\beta) \big) \hookrightarrow H^\infty \quad \text{contractively.} 
\end{equation}
Indeed, if $h \in \mathcal{M} \big( H^2 (\beta) \big)$, we easily get for all $w \in \D$:
\begin{displaymath} 
M_{h}^{\ast} (K_w) = \overbar{h(w)} K_w \, ;
\end{displaymath} 
so that taking norms and simplifying, we are left with $|h (w)|\leq \| M_h \|$, showing that $h \in H^\infty$ with $\| h \|_\infty \leq  \| M_h \|$. 
\goodbreak

\begin{proposition} \label{prop multipliers}
We have ${\cal M} \big( H^2 (\beta) \big) = H^\infty$ isomorphically if and only if $\beta$ is essentially decreasing.
\end{proposition}
\begin{proof}
The sufficient condition is proved in \cite[beginning of the proof of Proposition~3.16]{LLQR-comparison}. 
For the necessity, we then have $\| M_h \| \approx \| h \|_\infty$ for every $h \in H^\infty$ by the Banach isomorphism theorem. 
Now, for $m > n$ (recall that $e_n (z) = z^n$):
\begin{displaymath}
e_m (z) = z^{m - n} z^n = (M_{e_{m - n}} e_n) (z) \, ;
\end{displaymath}
so, since $\| M_{e_{m - n}} \| \leq C \, \| e_{m - n} \|_\infty = C$ for some positive constant $C$:
\begin{displaymath}
\beta_m = \| e_m \|^2 \leq C^2 \, \| e_n \|^2 = C^2 \, \beta_n \, . \qedhere
\end{displaymath}
\end{proof}

In \cite[Section~3.6]{LLQR-comparison}, we gave the following notion of \emph{admissible} Hilbert space of analytic functions.
\goodbreak

\begin{definition}\label{defadmis}
A Hilbert space $H$ of analytic functions on $\D$, containing the constants, and with reproducing kernels $K_a$, $a \in \D$, is said to be 
\emph{admissible} if: 
\begin{enumerate}
\setlength\itemsep {-0.05 em}

\item [$(i)$]  $H^2$ is continuously embedded in $H$;  

\item [$(ii)$] $\mathcal{M} (H) = H^\infty$; 

\item [$(iii)$] the automorphisms of $\D$ induce bounded composition operators on $H$;

\item [$(iv)$] $\displaystyle \frac{\Vert K_a\Vert_H}{\Vert K_b\Vert_H} \leq h \bigg(\frac{1 - |b|}{1 - |a|} \bigg)$ for $a, b \in \D$, where 
$h \colon \R^+\to \R^+$ is a non-decreasing function.
\end{enumerate}
\end{definition}

We proved in that paper that every weighted Hilbert space $H^2 (\beta)$ with $\beta$ non-increasing is admissible, under the additional hypothesis 
that the automorphisms of $\D$ induce bounded composition operators. In view of Theorem~\ref{theo CS T_a beta SO}, we get the following result.
\goodbreak

\begin{proposition} \label{prop admissible}
Let $\beta$ be a weight.
\begin{enumerate}[$1)$]

\item If $\beta$ is essentially decreasing, then we have $(i),(ii),(iii)$ in Definition~\ref{defadmis}.

\item  If $\beta$ is  slowly oscillating, then we have $(iv)$ in Definition~\ref{defadmis}.
\end{enumerate}
\end{proposition}

Let us give a different proof from the one in \cite{LLQR-comparison}.

\begin{proof}
$1)$ Let us assume that $\beta$ is essentially decreasing. Then item $(i)$ holds, as well as item $(ii)$, by Proposition~\ref{prop multipliers}. Item $(iii)$ is 
Theorem~\ref{theo CS T_a beta SO}. 

$2)$ Now we assume that $\beta$ is slowly oscillating. 

Let $0 < s < r < 1$. 

Without loss of generality, we may assume that $r, s \geq 1 / 2$. It is enough to prove:
\begin{equation} \label{enough}
\| K_r \|^2 \leq C \, \| K_{r^2} \|^2
\end{equation}
for some constant $C > 1$. Indeed, iteration of \eqref{enough} gives:
\begin{displaymath}
\| K_r \|^2 \leq C^k \, \| K_{r^{2^k} } \|^2
\end{displaymath}
and if $k$ is the smallest integer such that $r^{2^k} \leq s$, we have $2^{k - 1} \log r > \log s$ and $2^k \leq D \, \frac{1 - s}{1 - r}$ 
where $D$ is a numerical constant. Writing $C = 2^\alpha$ with $\alpha > 1$, we obtain:
\begin{displaymath}
\bigg( \frac{\| K_r \|}{\| K_s \|} \bigg)^2 \leq C^k = (2^k)^\alpha \leq D^\alpha \bigg( \frac{1 - s}{1 - r} \bigg)^\alpha \, .
\end{displaymath}

To prove \eqref{enough}, we pick some $M > 1$ such that $\beta_{2 n} \geq M^{- 1} \beta_n$ and $\beta_{2 n-1} \geq M^{- 1} \beta_n$,  for all $n \geq 1$, 
since $\beta$ is slowly oscillating. Write $t = r^2$. We have:
\begin{displaymath}
\| K_r \|^2 = \frac{1}{\beta_0} + \sum_{n = 1}^\infty \frac{t^{2 n}}{\beta_{2 n}} 
+ \sum_{n = 1}^\infty \frac{t^{2 n - 1}}{\beta_{2 n - 1}} \, \raise 1 pt \hbox{,}
\end{displaymath}
implying, since $t^{2 n - 1} \leq 4 \, t^{2 n}$:
\begin{displaymath}
\| K_r \|^2 \leq \frac{1}{\beta_0} + M \sum_{n = 1}^\infty \frac{t^{2 n}}{\beta_n} 
+ 4 M \sum_{n = 1}^\infty \frac{t^{2 n}}{\beta_n} \leq 5 M \| K_t \|^2 \, . \qedhere
\end{displaymath}
\end{proof}

The notion of admissible Hilbert space $H$ is useful for the set of conditional multipliers:
\begin{displaymath} 
\mathcal{M} (H, \phi) = \{w \in H \tq w \, (f \circ \phi) \in H \text{ for all } f \in H \} \, . 
\end{displaymath} 
As a corollary of \cite[Theorem~3.18]{LLQR-comparison} we get:
\begin{corollary}
Let $\beta$ be essentially decreasing and slowly oscillating. Then:
\begin{enumerate}
\setlength\itemsep {-0.05 em}

\item [$1)$] $\mathcal{M} (H^2, \phi) \subseteq \mathcal{M} \big(H^2 (\beta), \phi \big)$; 

\item [$2)$] $\mathcal{M} \big(H^2 (\beta), \phi \big) = H^2 (\beta)$ if and only if $\| \phi \|_\infty < 1$;
 
\item [$3)$] $\mathcal{M} \big(H^2 (\beta), \phi \big) = H^\infty$ if and only if $\phi$ is a finite Blaschke product.
\end{enumerate}
\end{corollary}

We add here as another application of our results an answer to a question appearing in Problem 5 in Zorboska's thesis \cite{ZorboPgD}.

\begin{theorem}
Let $\beta$ be a weight such that $H^2(\beta)$ is disc-automorphism invariant and a symbol $\varphi$ inducing a compact composition operator on $H^2(\beta)$. 
Then the Denjoy-Wolff point of $\varphi$ must be in $\D$.
\end{theorem}

In other words, $\varphi$ has a fixed point in $\D$.

In the statement, ``$H^2(\beta)$ is disc-automorphism invariant'' means that for all the automorphisms $T_a$, where $a\in\D$, $C_{T_a}$ is bounded on 
$H^2(\beta)$ (equivalently it is bounded for at least one $a\in\D\setminus\{0\}$). 

For the definition of the Denjoy-Wolff point, we refer to \cite{Shapiro-livre}.

\begin{proof}
From Theorem~\ref{theo CN T_a beta SO}, we know that $\beta$ is slowly oscillating, and from Proposition~\ref{prop admissible}, we know that
\begin{equation}\label{(iv)}
\qquad\quad\frac{\Vert K_a\Vert_{H^2(\beta)}}{\Vert K_b\Vert_{H^2(\beta)}} \leq h \bigg(\frac{1 - |b|}{1 - |a|} \bigg)\qquad\hbox{for every } a, b \in \D,
\end{equation}
where  $h \colon \R^+\to \R^+$ is a non-decreasing function.

Now we split the proof into two cases.\smallskip

If $\sum\dfrac{1}{\beta_n}<\infty$, then $H^2(\beta)\subset A(\D)$ (continuously) thanks to the Cauchy-Schwarz inequality. It follows from a result due to 
Shapiro \cite[Theorem~2.1]{ShapiroPAMS} that $\|\varphi\|_\infty<1$ and the conclusion follows obviously in this case. 
\smallskip

If $\sum\dfrac{1}{\beta_n}=\infty$, then the normalized reproducing kernel $\dfrac{K_z}{\|K_z\|}$ is weakly converging to $0$ when $|z|\to1^-$, since 
$\|K_z\|\rightarrow+\infty$.

Since $C_\varphi$ is compact,  $C_\varphi^\ast$ is compact as well and we get 
$$\dfrac{K_{\varphi(z)}}{\|K_z\|}\longrightarrow0\qquad\hbox{when }|z|\to1^- $$
and equivalently
$$\dfrac{\|K_z\|}{\|K_{\varphi(z)}\|}\longrightarrow+\infty\qquad\hbox{when }|z|\to1^-\;. $$
But, from \eqref{(iv)}, we get
$$ h \bigg(\frac{1 - |\varphi(z)|}{1 - |z|} \bigg)\longrightarrow+\infty\qquad\hbox{when }|z|\to1^- $$
hence, since $h$ is non-decreasing, $$\frac{1 - |\varphi(z)|}{1 - |z|}\longrightarrow+\infty\qquad\hbox{when }|z|\to1^-\,. $$
By the Denjoy-Wolff theorem \cite{Shapiro-livre}, the conclusion follows in this case too. 
\end{proof}

\goodbreak

\section {Appendix} 

In this appendix, we give the proof of Proposition~\ref{tihebr}. 
\smallskip

The following lemma can be found in \cite[Lemma~1, page 47]{Montgomery}. 

\begin{lemma} \label{spc} 
Let $F \colon [u,v] \to \R$, with $u < v$, be a ${\cal C}^2$- function with $F'' > 0$, and $F '$ not vanishing on $[u, v]$. Let 
\begin{displaymath}
J = \int_{u}^v e^{i F(x)} \, dx \, . 
\end{displaymath}
Then: 
\begin{itemize}
\setlength\itemsep {-0.1 em}

\item [{\rm a)}] if $F ' > 0$ on $[u, v]$, then $|J| \leq \frac{2}{ F ' (u)}\,$;

\item [{\rm b)}] if $F ' < 0$ on $[u, v]$, then $|J| \leq \frac{2}{|F ' (v)|} \, \cdot$
\end{itemize}
\end{lemma}
\goodbreak

\begin{proof} [Proof of Proposition~\ref{tihebr}] 

Write now the integral $I$ of Proposition~\ref{tihebr} on $[A, B]$ as $I = I_1 + I_2 + I_3$ with:
\begin{displaymath}
I_1 = \int_A^{c - \eta} \e^{i F(x)} \, dx \, , \quad  I_2 = \int_{c - \eta}^{c + \eta} \e^{i F(x)} \, dx \, , 
\quad  I_3 = \int_{c + \eta}^B \e^{i F(x)} \, dx \, .
\end{displaymath}
Lemma~\ref{spc} with $u = A$ and $v = c - \eta$ implies, since $F'<0$ on $[A,c-\eta]$:
\begin{equation} \label{ein} 
|I_1| \leq  \frac{2}{|F ' (c -  \eta)|} \leq \frac{2}{\eta \lambda_2} \, \raise 1 pt \hbox{,}  
\end{equation}
where, for the last inequality, we just have to write 
\begin{displaymath}
|F ' (c - \eta)| = F ' (c) - F ' (c - \eta) = \eta \, F '' (\xi) 
\end{displaymath}
for some $\xi \in [c - \eta, c]$ so that $F '' (\xi) \geq \lambda_2$. 
\smallskip

Similarly, Lemma~\ref{spc} with $u = c + \eta$ and $v = B$ implies
\begin{equation} \label{zwei} 
|I_3| \leq  \frac{2}{F ' (c + \eta)} \leq \frac{2}{\eta \lambda_2} \, \cdot 
\end{equation}

We can now estimate $I_2$. The Taylor formula shows that
\begin{displaymath}
F (x) = F (c) + \frac{(x - c)^2}{2} F '' (c) + R \, ,
\end{displaymath}
with
\begin{displaymath}
|R| \leq \frac{|x - c|^3}{6} \, \lambda_3 \, .
\end{displaymath}
Hence 
\begin{displaymath}
I_2 = \e^{i F(c)} \int_{0}^\eta 2 \exp \bigg( \frac{i}{2} \, x^2 F '' (c) \bigg) \, dx + S 
\end{displaymath}
with 
\begin{displaymath}
|S| \leq \lambda_3 \int_{0}^\eta \frac{x^3}{3} \, dx = \frac{\, \eta^4}{12} \, \lambda_3 \, .
\end{displaymath}

Finally, set 
\begin{displaymath}
K = \int_{0}^\eta 2 \exp \bigg( \frac{i}{2} \, x^2 F '' (c) \bigg) \, dx \, .
\end{displaymath}
We make the change of variable $x = \sqrt{\frac{2}{F '' (c)}} \, \sqrt{t}$. Recall that 
$\int_{0}^\infty  \frac{\e^{i t}}{\sqrt{t}} dt = \sqrt{\pi} \, \e^{i \pi / 4}$ is the classical Fresnel integral, and that an 
integration by parts gives, for $m > 0$: 
\begin{displaymath} 
\bigg| \int_{m}^\infty  \frac{\e^{i t}}{\sqrt{t}} \, dt \bigg|  \leq \frac{2}{\sqrt m} \, \cdot
\end{displaymath}
Therefore, with $m = \frac{\, \eta^2}{2} \, F '' (c)$:
\begin{displaymath}
K = \sqrt{\frac{2}{F '' (c)}} \int_{0}^m  \frac{\e^{i t}}{\sqrt{t}} \, dt 
= \sqrt{\frac{2 \pi}{F '' (c)}} \, \e^{i \pi/4} + R_m \, ,
\end{displaymath}
with
\begin{displaymath}
|R_m| \leq C \, \sqrt{\frac{1}{F '' (c)}} \, \frac{1}{\sqrt{m}} \leq \frac{C}{\eta \lambda_2} \, \cdot
\end{displaymath}

All in all, we proved that 
\begin{equation} \label{aia} 
I_2 = \sqrt{\frac{2 \pi}{F '' (c)}} \, \exp \big[ i (F (c) + \pi / 4) \big] + O\, \bigg( \frac{1}{\eta \lambda_2} + \eta^4 \lambda_3 \bigg) \, \cdot
\end{equation}
and the same estimate holds for $I$, thanks to \eqref{ein} and \eqref{zwei}.
\smallskip

We have hence proved Proposition~\ref{tihebr}. 
\end{proof}
%

\bigskip

\noindent{\bf Acknowledgements.}\ L. Rodr{\'\i}guez-Piazza is partially supported by the projects PGC2018-094215-B-I00 and  PID2022-136320NB-I00 
(Spanish Ministerio de Ciencia, Innovaci\'on y Universidades, and FEDER funds). 
Parts of this paper was made when he visited the Universit\'e d'Artois in Lens and the Universit\'e de Lille in January 2020 and September 2022. It is his 
pleasure to thank all his colleagues in these universities for their warm welcome. 

Parts of this paper were made during an invitation of the second-named and third-named authors by the IMUS at the Universidad de Sevilla; it is their pleasure to 
thank all people in that University who made this stay possible and very pleasant, and especially Manuel Contreras. 

The third-named author was partly supported by the Labex CEMPI (ANR-LABX-0007-01). 

This work is also partially supported by the grant ANR-17-CE40-0021 of the French National Research Agency ANR (project Front).

We warmly thank R.~Zarouf for useful discussions and informations. 

\goodbreak

\smallskip

{\footnotesize
Pascal Lef\`evre and Daniel Li \\
Univ. Artois,  UR~2462, Laboratoire de Math\'ematiques de Lens (LML), F-62\kern 1mm 300 LENS, FRANCE \\
pascal.lefevre@univ-artois.fr -- daniel.li@univ-artois.fr
\smallskip

Herv\'e Queff\'elec \\
Univ. Lille Nord de France, USTL,  
Laboratoire Paul Painlev\'e U.M.R. CNRS 8524 \& F\'ed\'eration Math\'ematique des Hauts-de-France FR~2037 CNRS, 
F-59\kern 1mm 655 VILLENEUVE D'ASCQ Cedex, FRANCE \\
Herve.Queffelec@univ-lille.fr
\smallskip
 
Luis Rodr{\'\i}guez-Piazza \\
Universidad de Sevilla, Facultad de Matem\'aticas, Departamento de An\'alisis Matem\'atico \& IMUS,  
Calle Tarfia s/n  
41\kern 1mm 012 SEVILLA, SPAIN \\
piazza@us.es
}

\end{document}